\documentclass{article}
 \usepackage[T1,T2A]{fontenc}
\usepackage[unicode]{hyperref}
\usepackage{amsmath,amssymb,amsthm,orcidlink,booktabs,pdflscape}
\usepackage[margin=1in]{geometry}
\usepackage{CJKutf8}
\usepackage[final]{microtype}
\title{Acyclic Poset Multiplihedra and their Quotients}
\author{Stefan Forcey\textsuperscript{\orcidlink{0000-0003-2251-5710}}\footnote{Department of Mathematics, University of Akron, Akron, Ohio 44325, United States of America}\and Ross Glew\textsuperscript{\orcidlink{0000-0002-8100-928X}}\footnote{Department of Mathematics and Theoretical Physics, University of Hertfordshire, Hatfield, Hertfordshire\ \textsc{al10 9ab}, United Kingdom}\and Chris Tapo\textsuperscript{\orcidlink{0000-0000-0000-0000}}\footnotemark[1]
\\[1em]
\href{mailto:sforcey@uakron.edu}{\texttt{sforcey@uakron.edu}}
\hskip1em
\href{mailto:r.glew@herts.ac.uk}{\texttt{r.glew@herts.ac.uk}}
\hskip1em
\href{mailto:cet64@uakron.edu}{\texttt{cet64@uakron.edu}}
}
\usepackage[french,german,portuguese,italian,latin,russian,ukrainian,british]{babel}
\usepackage[osf]{newpxtext}
\usepackage[varbb,slantedGreek]{newpxmath}
\usepackage[artemisia]{textgreek}
\usepackage[utf8]{inputenc}
\usepackage{tikz-cd}
\usetikzlibrary{babel}
\newcommand{\precdot}{\prec\mathrel{\mkern-5mu}\mathrel{\cdot}}

\begin{document}
\maketitle
\newtheorem{theorem}{Theorem}
\newtheorem{lemma}{Lemma}
\newtheorem{conjecture}{Conjecture}
\newtheorem{corollary}{Corollary}
\theoremstyle{definition}
\newtheorem{defn}{Definition}
\theoremstyle{remark}
\newtheorem{example}{Example}
\newcommand\cyrillic[1]{\fontfamily{Domitian-TOsF}\selectfont \foreignlanguage{russian}{#1}}
\begin{abstract}
Two families of polytopes underlie the combinatorics of associative operations. Associahedra and multiplihedra respectively capture the information in the operation itself and in the morphisms that respect that operation. The first applications of these polytopes, from Stasheff, were for modeling homotopy associative spaces and their homotopy homomorphisms. Later, lax and weak higher categories used both as the shapes of commuting diagrams. More recently, the associahedra have been generalized to versions based on graphs, and then to acyclic versions based on posets: the acyclonestohedra. Meanwhile, the associated multiplihedra have also been generalized to graph associahedra and generalized permutohedra. Here we complete that picture: we use the graph multiplihedra to define and realize new acyclic poset multiplihedra for the acyclic poset associahedra. Quotients of these are also studied, concluding with the conjectured interval polytopes of the order polytopes of posets.  
\end{abstract}

\baselineskip=17pt

\begin{figure}[h]
    \centering
     \includegraphics[width=0.65\linewidth]{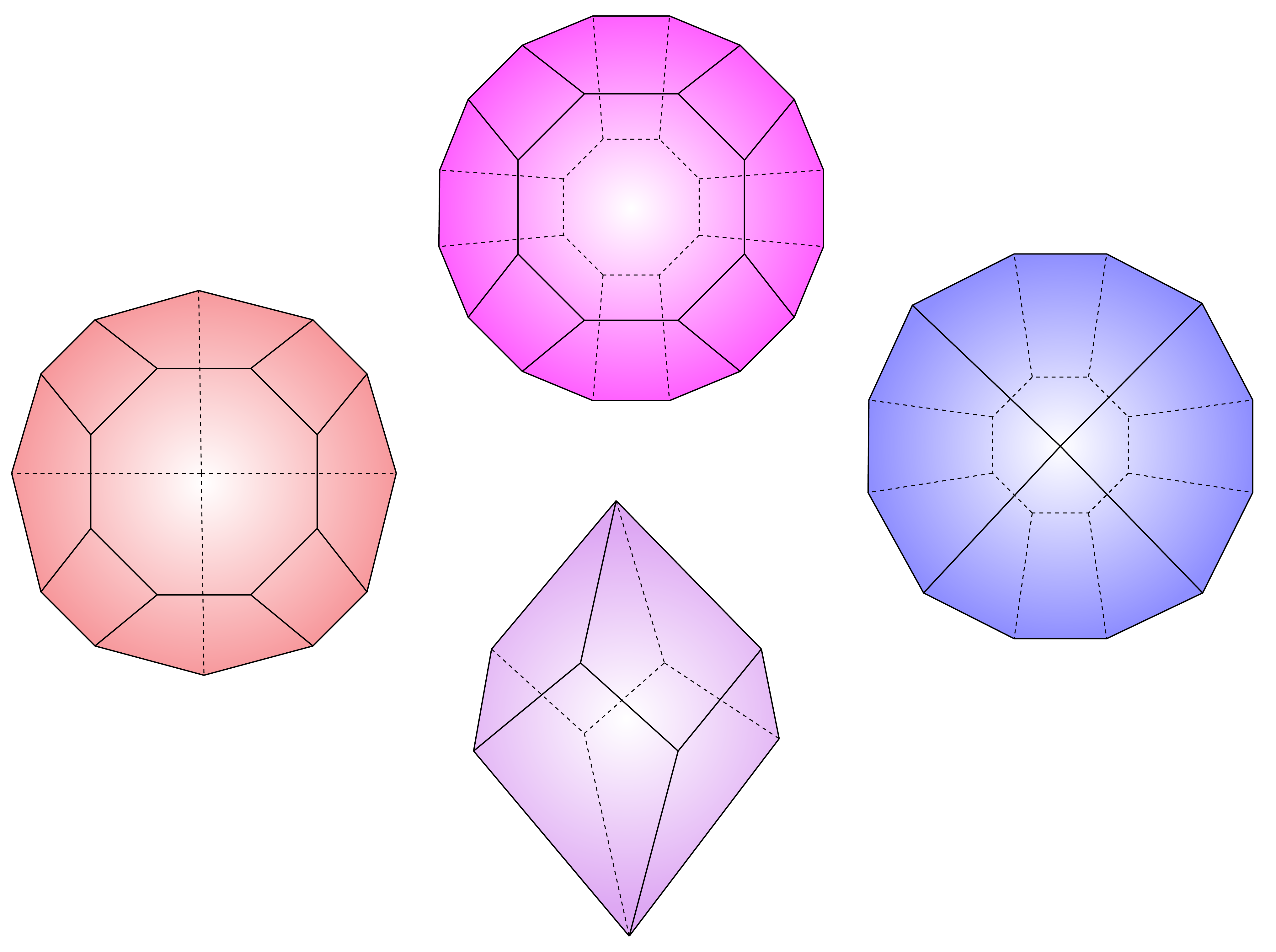}
    \caption{Cast of characters: the acyclic poset multiplihedron and its quotients for the bowtie poset. Figure~\ref{fig:bowtiemult} shows some of the face labels, while Table~\ref{tab:bow} gives a geometric realization of the multiplihedron.}
    \label{fig:topp}
\end{figure}

\section{Background and Overview}
\textit{Tubings, nestings, and pipings} are special subsets of a power set $\mathcal{P}(S)$ of a given \textit{ground} set $S$. The definitions of each begin with the elements of the power set they may contain: the \textit{tubes, nests, or pipes}. Often that initial collection of subsets $\mathcal{B}\subseteq \mathcal{P}(S)$ is called a \textit{building} set. The axioms for a building set $\mathcal{B}$ are: 1) the empty set is not a member of $\mathcal{B}$; 2) the building set  $\mathcal{B}$ must include all the singletons of $S$; and 3) for any two elements of $\mathcal{B}$, if they have nonempty intersection, then their union is also in $\mathcal{B}.$ In this paper we will deal with tubes of a graph as defined in \cite{cd}, which comprise a building set on the nodes of that graph. Pipes as defined in \cite{galashin} are special subposets of elements of a poset. Pipes do not make up a building set of the poset elements; but instead they are a special class of tubes of the line graph of the Hasse diagram of the poset. The term nest is used in \cite{laplante-anfossi} for what turn out to be  pipes, but for a poset that has a Hasse diagram that is a tree, or forest. Historically then: tubes of a graph were defined first, then nests and pipes of a poset were described independently, then nests were realized to be a specialization of pipes, specifically in \cite{defant2024operahedronlattices}, and finally, it was pointed out in \cite{mpp} and \cite{mppfull} that pipes and pipings correspond precisely to special versions of graph tubes and tubings. 

We will go over the definitions in the next section. But since all three sorts of subset are in the class of tubes of a graph, we will often call them tubes for ease of the discussion, leaving to the context what special sort of tube they are. Thus a tubing, nesting, or piping is a collection of tubes obeying the relevant rules. The first rule, obeyed by all three varieties, is that pairs of tubes in that collection must either be nested or disjoint: if they intersect then one must be a subset of the other. Thus the tubes in a tubing (or piping or nesting) are partially ordered by inclusion, and indeed the resulting poset formed by any tubing, nesting, or piping possesses a tree (or forest) as its Hasse diagram. If the elements of a tubing could overlap, that would allow a cycle in the underlying graph of the Hasse diagram.\footnote{Allowing overlap of tubes, up to further limitations, moves us into the realm of multitubings and multiassociahedra.}

The fact that the Hasse diagram of a tubing is always treelike allows certain superstructures to be added. First, the Hasse tree can be  \textit{painted} with colors (usually two main colors). The earliest examples of these superstructures are on the nestings of the linear poset (tubings of the path graph) which naturally label the faces of Stasheff's associahedra \cite{stasheff}. The 2-colored superstructure on these tubings make up the polytopes called multiplihedra \cite{forcey}. 2-colored structure is generalized to graph tubings in \cite{dev-forc}, and further to several additional classes of generalized permutohedra in \cite{doker}, where the results are called lifted generalized permutohedra. Multiplihedra are also seen as shuffles of the generalized permutohedra and generalized to $n$ colors in \cite{shuffle}. 
\begin{figure}
    \centering
    \includegraphics[width=\linewidth]{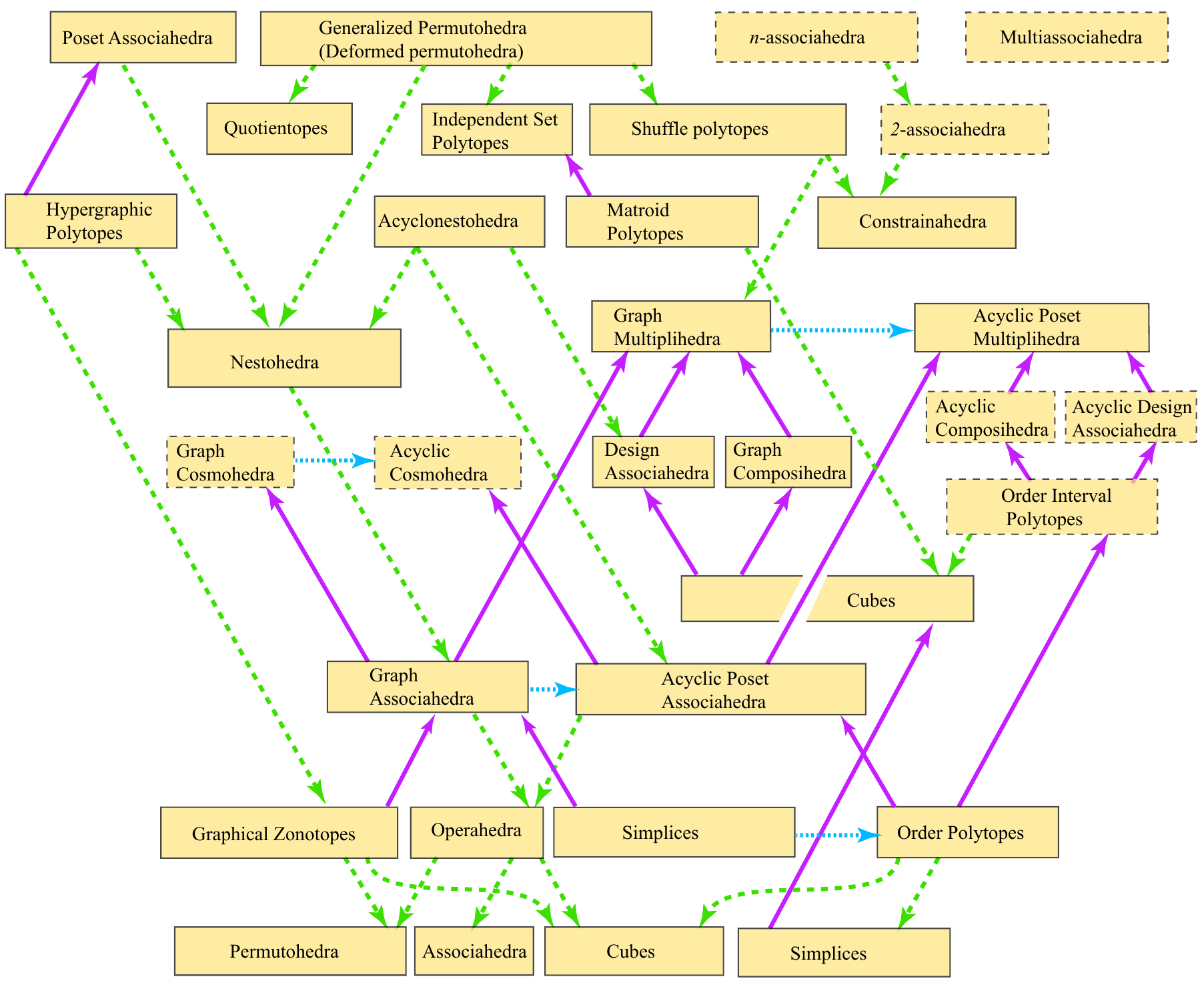}
    \caption{Relationships of families of polytopes: dashed green arrows slope down to indicate subset containment: lower contained in upper. Solid violet arrows point up to indicate lifting: refining the fan, truncation, taking intervals, or extending combinatorial structure. Dashed horizontal arrows show taking a cross section, left to right. Dashed boxes are conjectural polytopes. Many families and relationships were too hard to show here, such as the facts that interval hypergraphic polytopes include the associahedra, and that multiassociahedra include the associahedra, simplices, and certain kinds of amplituhedra. \\ $\text{~~~~~~~}$ For details on the relationships shown (and not shown) sources include: for hypergraphic polytopes \cite{berg, pilaud_berg, pilaud-sack}, for poset associahedra \cite{dfrs}, for generalized permutohedra and nestohedra \cite{postnikov}, for quotientopes \cite{quotient},  for graph multiplihedra and their quotients \cite{dev-forc}, for generalized multiplihedra, matroid polytopes and independent set polytopes \cite{doker}, for acyclonestohedra \cite{mppfull}, for operahedra \cite{laplante-anfossi, defant2024operahedronlattices}, for acyclic poset associahedra and order polytopes \cite{galashin}, for 2-associahedra and $n$-associahedra \cite{bottman, bottn}, for  shuffle polytopes and constrainahedra \cite{shuffle, bottc}, for cosmohedra and graph cosmohedra \cite{Arkani-Hamed:2024jbp,Glew:2025otn, ardilacosmo}, for acyclic cosmohedra \cite{forc-glew}, and for multiassociahedra \cite{multiassoc}. }
    \label{fig:polymapmess}
\end{figure}

The second type of superstructure, which we only briefly mention here, is when the graph of the Hasse diagram of each tubing can be used to define a new building set for an iterated construction of nesting: using pipes (also known as nests, since it is a tree.) These 2-fold iterations of the nestings of the path graph make up the cosmohedra, polytopes described in \cite{Arkani-Hamed:2024jbp} and \cite{ardilacosmo}, and generalized to all graphs and posets in \cite{Glew:2025otn} and \cite{forc-glew}. 

In this paper we focus on the first type, the colored superstructure. We describe a 2-colored version of the poset associahedra of Galashin. These latter were originally defined in \cite{galashin}, and that paper refers to them as $P$-associahedra, to differentiate them from the poset associahedra defined in \cite{dfrs}.  Realizations for Galashin's poset associahedra were found in \cite{sack}, \cite{mpp}, and \cite{mppfull}. In those three sources they are referred to as poset associahedra.  However, there is a clear distinction between the two types of poset associahedra: the pipes (and pipings) of Galashin's poset associahedra are convex (and \textit{acyclic}), while the original poset associahedra have tubes that are \textit{lower sets}, or order ideals, with graph associahedra as a special case. We will refer to Galashin's version as \textit{acyclic poset associahedra}, which is also appropriate considering that they are special \textit{acyclonestohedra}. \footnote{Neither sort of pipes or tubes on posets make a building set on the elements of the poset. In fact they both generalize the building set property in the same way: rather than satisfying the condition that for intersecting tubes their union is found in $\mathcal{B}$, they satisfy the condition that there exists a tube in $\mathcal{B}$ containing that union.} The polytopes we consider in this paper are mapped out in a much broader context in Figure~\ref{fig:polymapmess}.

\subsection{New polytopes}
In the next section we will review colored graph tubings, which have blue, purple and red tubes, and obey the rule that only blue tubes can be nested inside of blue or purple tubes. These colored tubings form a poset structure: the face poset of a polytope, the graph multiplihedron, as we will review in Sections~\ref{rev1} and~\ref{vert}. Then in Section~\ref{rev3} we review order polytopes and acyclic poset associahedra.
Our main result in Section~\ref{new} is that for any poset $P$ the acyclic colored tubings also constitute a polytope: that is, the poset of colored pipings of $P$ is equivalent to a face poset of a polytope, the acyclic poset multiplihedron.
The proof of this theorem, in Section~\ref{proofy}, will be simplified by the fact that we can rely on earlier results about the graph multiplihedron. Our new acyclic poset multiplihedra will be realized as cross sections of the graph multiplihedra, and inherit all the nested structure from the tubing structure of the latter. In Section~\ref{conjy} we define quotients on the the face poset of the acyclic poset multiplihedra, and show that the end result of these is to capture the intervals of the poset of faces of the order polytope. This leads to a conjecture that the quotients can actually be realized, and we present a candidate for this geometric realization which works in low dimensions.

\section {Review of graph associahedra, multiplihedra, and acyclonestohedra}
\label{rev}

For a connected graph $G$, a \textit{tube} $t$ is a connected induced subgraph, often given by its set of nodes. Tubes are drawn using closed curves around their nodes when the graph is small enough that the subgraph is easily seen to be thus enclosed. These tubes are uncolored for the graph associahedra, or colored (thick red, thin blue, or dashed purple) for the multiplihedra. 

A \textit{tubing} $T$ is a set of tubes of $G$, each pair nested or
disjoint, for which all unions of the tubes in $T$
are induced subgraphs. A \textit{colored tubing} is a tubing of colored tubes, obeying the rule that only blue tubes can be nested inside of blue or purple tubes.

The original motivation for two colors in the combinatorial structures of multiplihedra was to denote domain and range for a function or functor; blue for domain and red for range. We have used actual colors: red, blue, and purple lines---but the red is thick, blue thin, and purple dashed for printing in black and white. (The definitions from \cite{dev-forc} use the terms thick, thin, and broken for our red, blue, and purple, and refer to the colored tubing as a marked tubing.)

The definitions are subtle, and it can be helpful to note some consequences. For instance the rule that unions must be induced subgraphs means that two tubes $s,t$ both in a tubing $T$ cannot have tube $s$ contain one node of a graph edge while a disjoint tube $t$ contains the other. The rule that only blue tubes can be nested inside of non-red tubes means that purple tubes never have purple tubes nested inside them. 

For a graph $G$ with multiple connected components, we sometimes add that the entire graph $G$ is still considered a tube (despite being disconnected) and thus that the connected components of $G$ cannot all be members of a tubing at once. Alternatively for disconnected graphs we could define the tubings to be any collection of tubings on the connected components; this gives a Cartesian product of tubings on the components. For the rest of this paper we will restrict our attention to the connected graphs, leaving the choice of how to describe disconnected graphs up to practitioner. Let the connected graph $G$ have $n$ nodes.
\begin{defn}
The \textit{graph associahedron} $\mathcal{K}(G)$ is defined in \cite{cd} as the $(n-1)$-dimensional polytope whose face poset corresponds precisely to the uncolored graph tubings, ordered by reverse inclusion. 
\end{defn}
\begin{defn}\label{multdef}
    The \textit{graph multiplihedron} $\mathcal{J}(G)$ is defined in \cite{dev-forc} as the $n$-dimensional polytope whose face poset  corresponds to all the valid colored tubings of a graph. The ordering is given by the following four covering relations $U\precdot V$: tubing $U$ is less than (covered by) tubing $V$ if:
\end{defn}

\begin{enumerate}
    \item $U$ is achieved by adding a single blue tube to $V$, which must be added inside another blue or purple tube. 
    \item $U$ is achieved by adding a single red tube to $V$, which must be added inside another red tube.
    \item $U$ is achieved from $V$ by \textit{resolving} a purple tube of $V$, by making it blue or red.
    \item $U$ is achieved from $V$ by resolving a purple tube of $V$ to red, and simultaneously adding one or more purple tubes just inside of that new red tube (and not inside any other tube).
\end{enumerate}

\begin{figure}
    \centering
    \includegraphics[width=.75\linewidth]{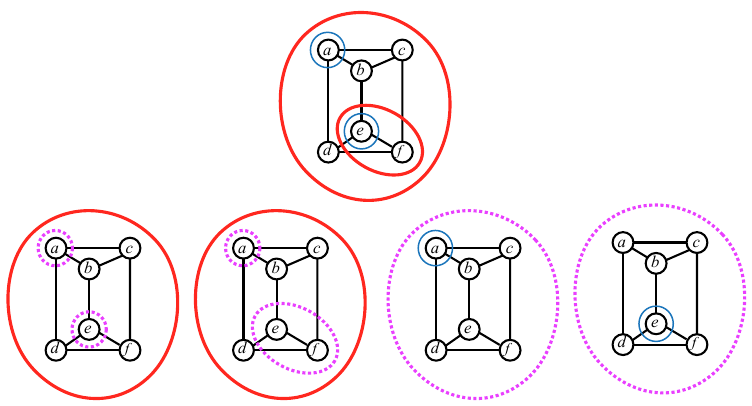}
    \caption{A colored tubing for a 2D face of a 6D graph multiplihedron, top, and four of the 5D facets that contain that 2D face. All four of these facet tubings are induced by acyclic pipings of the double-claw poset in Figure~\ref{fig:runonecorre}. Vertices of the graph multiplihedron of this 2D face are shown in Figure~\ref{fig:runonenum}.} 
    \label{fig:runonefacet}
\end{figure}

Relations in the poset of colored tubings thus generated can be difficult to determine. Simply being subsets of the tubes (of each color) does not make one colored tubing comparable to another. For instance, starting with a single universal red tube, the only tubings less than that one must have all red tubes, since there is no way to legally add blue or purple tubes. The same is true for the single universal blue tube: and these two examples are facets of the resulting polytope that are in fact equivalent to the graph associahedron for the same graph. 

 The minimal elements are those tubings with a maximal number of tubes, some red and some blue tubes, no purple. The maximal number of tubes is always $n$, the number of nodes.     Thus we can associate an element of the poset with the maximal nestings that it can be extended to become, so that we have $U\prec V$ when the  $n$-tube tubings that $U$ can become is a subset of those that $V$ can become. The maximal element of the entire poset is the tubing with one universal purple nest. Then by adding a single blue tube inside it, or resolving it to blue, or by resolving it to red,  and perhaps adding any allowable purple tubes inside it, we see the penultimate elements of the poset which correspond to facets of our face poset. 

A unifying principle behind the relations we have listed is that they describe the resolution of a colored tree into a binary tree, with simplest coloring.  Ardila and Doker point out in \cite{doker} that the Hasse diagram of a tubing, a tree, can have its edges painted with two colors. Red at the root, and blue at the leaves, the rule for painting is that blue must be above red: purple tubes correspond to nodes where red and blue meet at a branch point of the tree. Then the covering relations can be expressed as resolving the tree until it is binary with no bi-colored branch points. 

\subsection{Multiplihedron facet realization}\label{rev1} Let $G$ be a connected graph with vertices $v_1,\dots,v_n.$ It is shown in \cite{dev-forc} that $\mathcal{J}(G)$ is the poset of faces of an $n$-dimensional convex polytope with vertices corresponding to the colored tubings of $n$ tubes, where $n$ is the number of nodes of $G.$
In \cite{dev-forc} is given a geometric realization of the graph multiplihedron  for any graph $G$ with $n$ nodes. The polytope is full dimensional, in $\mathbb{R}^n$ where coordinate $x_i$ corresponds to node $v_i$ of $G$. The facet inequalities for the polytope come in two flavors:
\begin{enumerate}
    \item For each colored tubing of $G$ with only one blue tube $t$, a lower facet given by $$\sum_{v_i\in t} x_i \ge 3^{|t|-1} .$$
    \item For each colored tubing of $G$ with only one red tube $G$ itself, and (possibly) $k$ purple tubes $t_1\dots t_k$, an upper facet given by: $$\sum_{v_i\notin \cup t_j}x_i \le 3^n-\sum_{j=1}^k3^{|t_j|}.$$
\end{enumerate}

To see examples, consider the four facet tubings in Figure~\ref{fig:runonefacet}. The variables $x_a,x_b,\dots,x_f$ correspond to the nodes. The first two tubings yield upper inequalities $x_b+x_c+x_d+x_f \le 729-3-3$ and $x_b+x_c+x_d \le 729-3-9.$ The second two yield lower inequalities $x_a \ge 1$ and $x_e \ge 1.$ 

\subsection{Multiplihedron vertices}\label{vert} In \cite{dev-forc} it is shown that the vertices of this polytope realization of the graph multiplihedron are given by the following assignment of values to the coordinates $x_i$ for an $n$-tube tubing $T$ with all red and blue tubes (thus maximal as a tubing, so minimal in the poset of tubings). For a node $v_i$ of $G$ there will be a tube $\epsilon(v_i)$ that is the smallest containing $v_i$. For any tube $t$ there may be some tubes $s$ that are nested inside $t$ and inside no other tube: we say $s\precdot t$ in the containment order. 
\begin{enumerate}
    \item 
 For $\epsilon(v_i)$ a blue tube we have:
$$x_i= 3^{|\epsilon(v_i)-1|} - \sum_{s\precdot \epsilon(v_i)}3^{|s|-1}.$$
\item For $\epsilon(v_i)$ a red tube we have an extra factor of 3:
$$x_i= 3^{|\epsilon(v_i)|} - \sum_{s\precdot \epsilon(v_i)}3^{|s|}.$$
\end{enumerate}

Example calculations of some coordinates associated to nodes of a graph with a maximal colored tubing are shown in Figure~\ref{fig:runonenum}. Notice how these coordinate values obey the four facet inequalities (as equalities) from the facet tubings in Figure~\ref{fig:runonefacet}.

\subsection{Review of order polytopes and acyclic poset associahedra}\label{rev3}

We will work with connected posets $P,$ which means that the Hasse diagram of $P$ is a connected graph. A convex subposet is a subposet $C$ such that if $a\in C$ and $b\in C$, with $a\le x \le b$, then $x\in C.$ We often call a convex subposet \textit{acyclic}, since a non-convex subposet has the property that collapsing it to a single element in the Hasse diagram results in the remaining directed edges of the Hasse diagram forming a directed cycle, after the collapse. Similarly, we say a collection of subposets is acyclic if simultaneously collapsing all subposets in the collection results in no directed cycles.

\begin{defn}
    We consider partitions of the $n$ elements of the connected poset $P$ into connected parts. These connected-part partitions which are acyclic (and thus must have acyclic parts) label the faces of an $(n-2)$-dimensional polytope called the \textit{order polytope} $\mathcal{O}(P).$ The face poset of acyclic partitions is ordered by reverse refinement, with the partition into all singletons as the interior of the polytope and the trivial partition not included as a face: it can be considered the minimal element in the face lattice.
\end{defn}
 Figures~\ref{fig:trunc1} and ~\ref{fig:newby1real} show examples. $\mathcal{O}(P)$ was first defined by Stanley for posets with unique maximum and minimum elements in \cite{stanley}, and more generally for all posets by Galashin in \cite{galashin}.  Galashin's order polytopes are found combinatorially as faces of Stanley's, via adjoining a max and min if needed.
 Following Stanley, Galashin gives a facet realization for $\mathcal{O}(P)$  for real variables $x_i$ indexed by the $n$ elements $i\in P$, with  
 \begin{enumerate}
     \item facet inequalities: $x_i\le x_j$ for each covering relation $i\precdot j$ in $P$,
     \item equalities: $\displaystyle{\sum_{i\in P}x_i = 0}$ and  $\displaystyle{\sum_{i \precdot j}(x_j-x_i) = 1.}$
 \end{enumerate}
 Alternatively, there is a realization for $\mathcal{O}(P)$  implied in \cite[Sec. 8.2]{mppfull} for  variables $x_{e}$ indexed by the $|E|$ edges of the Hasse diagram of $P$ (covering relations of $P$), as a cross section of a simplex. 
 \begin{enumerate}
     \item inequalities: $x_{e} \ge 0$ for each edge $e$ of the Hasse diagram,
     \item equality: $\displaystyle{\sum x_{e} = 1.}$
     \item equalities for cycles $c$ in the underlying graph of the Hasse diagram of $P$:  $\displaystyle{\sum_{e \text{ with }c }x_e = \sum_{e \text{ against } c}x_e }.$  Choosing an arbitrary direction to traverse the cycle $c$, the orientation (covering relation) of the edge $e$ in $c$ will either agree with or be against the direction of traversal of $c.$
 \end{enumerate}


\begin{figure}
    \centering
    \includegraphics[width=\linewidth]{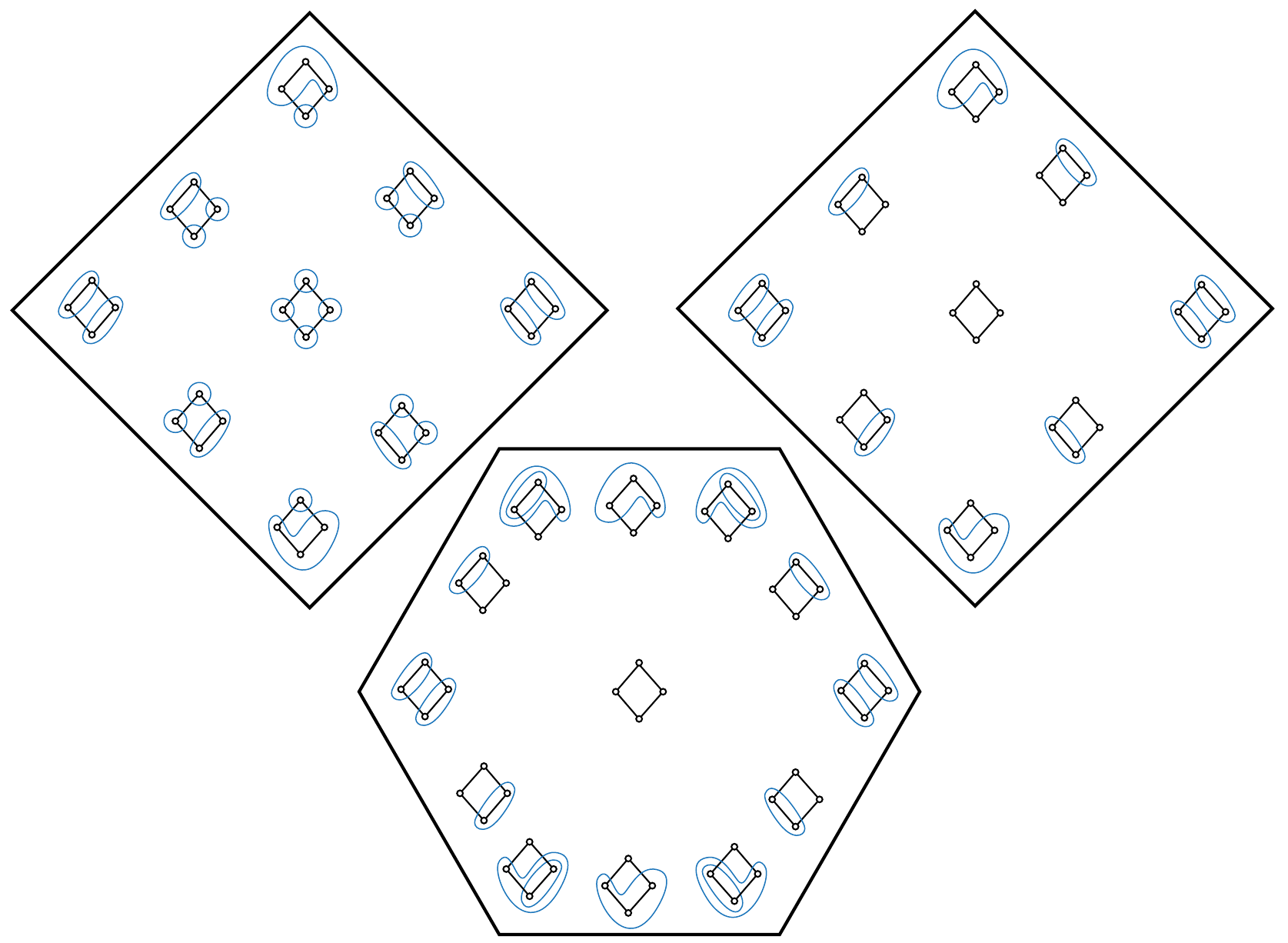}
    \caption{On the left is the order polytope $\mathcal{O}(P)$ for the diamond poset. On the right  is the same order polytope, showing only subposets of size greater than one. Two of the vertices are truncated, thus promoted to new facets of the acyclic poset associahedron in the center.}
    \label{fig:trunc1}
\end{figure}

For a connected poset $P$ we consider \textit{pipings} as described in \cite{galashin}.
A \textit{pipe} $t$ of $P$ is a connected convex subposet of $P$ with at least two nodes. That is the same as a connected convex subgraph of the Hasse diagram of $P$ with at least one edge. A piping $\mathcal{N}$ is a set of pipes such that each pair is nested or disjoint, and such that the collection of those pipes is also acyclic. Notice that unlike tubings, pipings can have adjacent pipes. Indeed the faces of the order polytope  $\mathcal{O}(P)$ are labeled by pipings: specifically the pipings that can be seen as partitions of the poset, with no nesting and from which we have discarded the singleton parts.

Keeping all the pipings of a poset, not just the partitions but also the nested pipings, corresponds to the following polytope defined in \cite{galashin}:
\begin{defn}
    For a poset $P$ the acyclic poset associahedron $\mathcal{A}(P)$ is an $(n-2)$-dimensional polytope whose face poset is the same as the collection of pipings of $P$, ordered by reverse inclusion.
\end{defn}
 See Figure~\ref{fig:trunc1} for the full poset of pipings on a diamond poset, and Figure~\ref{fig:newby1} for the facet pipes on a poset with five elements. The lower dimensional faces are then labeled by the piping that is found by collecting the pipes for the facets containing that lower dimensional face.

 It is shown in \cite{galashin} that the polytope $\mathcal{A}(P)$ exists as a truncation of the order polytope $\mathcal{O}(P)$. Truncating a face promotes it to a new facet. The full iterative procedure is described in \cite{galashin} in terms of stellations of the polar dual polytope. The stellations, or dually the truncations, must proceed in order of  dimension of the face to be stellated or truncated. 
The promoted faces are those for which the adjacent facets of $\mathcal{O}(P)$ are labeled by edges of the Hasse diagram which can together constitute a single pipe. See Figure~\ref{fig:trunc1} for the process on a polygon, and Figure~\ref{fig:newby1real} for the order polytope and the realization (via facet inequalities given in the next section) that exhibit the truncations on the order polytope which result in the example shown in Figure~\ref{fig:newby1}.

\begin{figure}
    \centering
    \includegraphics[width=.85\linewidth]{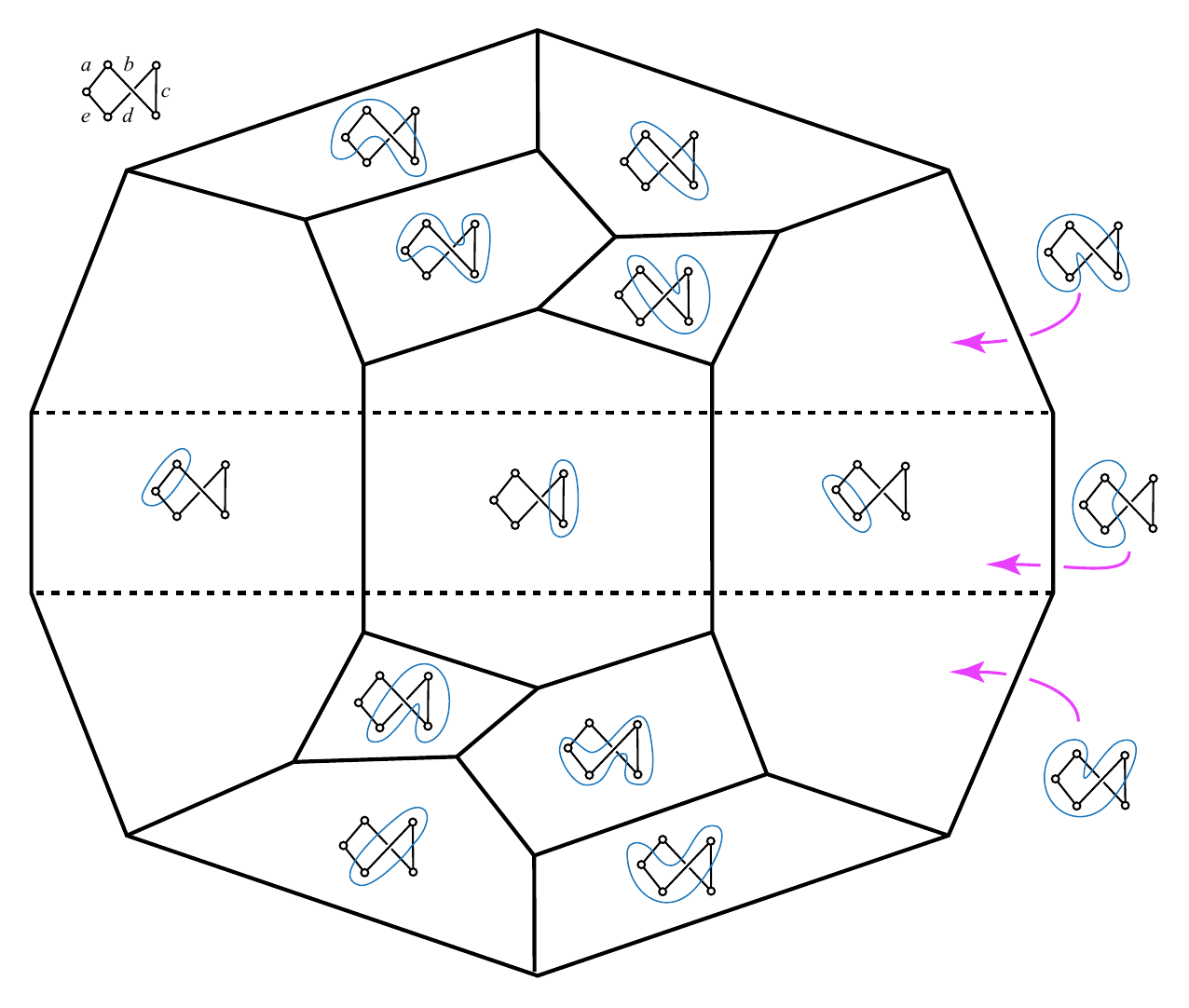}
    \caption{Acyclonestohedron. Specifically, the acyclic poset associahedron for the poset with Hasse diagram shown. Facets are labeled by pipes: tubes that are connected convex subposets. Figure~\ref{fig:newby1real} shows the order polytope for this poset, and the geometric realization of this acyclic poset associahedron.}
    \label{fig:newby1}
\end{figure}

In \cite{mppfull} the acyclonestohedron is defined more generally, and for a poset the acyclonestohedron is the (acyclic) poset associahedron.
Notice that a piping does not have to obey the requirement that the union of its pipes is an induced graph, in fact the edges of the derived graph achieved by collapsing the piping often result from an edge (a covering relation of $P$) between two pipes of the piping. However, we have the following from \cite{mppfull}: when we take the line graph $L(P)$ of the Hasse diagram\footnote{The line graph has a node for each edge (covering relation) of the Hasse diagram of $P$ and an edge between each pair that are coincident, both involving the same element of $P$.}, each pipe of $P$ corresponds to an \textit{induced} tube of $L(P)$: the edges in the pipe are the nodes in that tube. Moreover, a piping of the poset becomes a tubing of the line graph, via the direct translation: if an edge of the Hasse diagram is in a pipe of the piping, then the corresponding node of the line graph is in the corresponding tube of the induced tubing. The tubes will be compatible, either nested or non-adjacent due to the line graph construction. We show an example (with colored tubes) in Figure~\ref{fig:runonecorre}.  In fact, this leads to a main result in \cite{mppfull}: the authors discovered that  the acyclic poset associahedron $\mathcal{A}(P)$ for is realized as a cross section of the graph associahedron  $\mathcal{K}(L(P))$  for the line graph of the Hasse diagram.

\section {Multiplihedra for acyclic poset associahedra.}
\label{new}

Using the correspondence between pipings and induced acyclic tubings, it is straightforward to define the 2-colored version of the acyclonestohedron for a poset $P$: the acyclic poset multiplihedron. 
\begin{defn}
    For a poset $P$ we define the \textit{colored pipings} on $P$ as pipings with blue, red, and purple pipes, such that the induced colored tubing on the line graph of the Hasse diagram is a colored tubing of that line graph. Thus all the same requirements hold: a colored piping is just a special colored tubing, and the colored pipings are ordered via the same relations of adding or resolving colored tubes as listed in Definition~\ref{multdef}.
\end{defn}
 For example, we show a colored piping and its induced colored tubing on the line graph, in Figure~\ref{fig:runonecorre}. The poset of colored pipings for a poset $P$ is called the acyclic poset multiplihedron $\mathcal{J}(P)$.

 \begin{figure}[h!]
    \centering
    \includegraphics[width=0.5\linewidth]{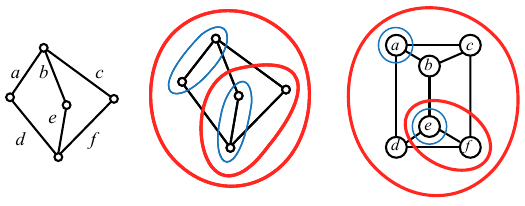}
    \caption{A colored piping on the poset with Hasse diagram at left, and its induced colored tubing on the line graph of that Hasse diagram, on the right. Some facet tubings induced by pipings that contain this piping are shown in Figure~\ref{fig:runonefacet}, and several maximal tubings with coordinates calculated are shown in Figure~\ref{fig:runonenum}.}
    \label{fig:runonecorre}
\end{figure}

\begin{theorem}
    The colored pipings of a poset $P$ with $n$ nodes correspond to the face poset of a convex polytope $\mathcal{J}(P)$ of dimension $n-1$.
\end{theorem}

\begin{proof}
    The proof is via an explicit geometric realization, which we will describe and prove in the next section. The summary is that the acyclic poset multiplihedron is a geometric cross section of the graph multiplihedron of the line graph of the Hasse diagram.
\end{proof}

 For examples: the vertices and some faces of the acyclic multiplihedron polytope for the bowtie poset are shown in Figure~\ref{fig:bowtiemult}, also pictured in Figure~\ref{fig:topp}. The acyclic multiplihedron for the diamond poset with labels is shown in Figure~\ref{fig:diamondmulti}.

Recall that for a Hasse diagram that is a tree, with no cycles, the acyclic poset associahedron is called the operahedron \cite{laplante-anfossi}. Since there are no cycles in the Hasse diagram of $P$, the cross section is trivial:  $\mathcal{A}(P)$ is combinatorially equivalent to  $\mathcal{K}(L(P))$, the graph associahedron of the line graph of the Hasse diagram. Similarly we have the following:
 \begin{corollary}
     A poset $P$ whose Hasse diagram has no cycle (every edge is a bridge) has an acyclic multiplihedron which is the graph multiplihedron of the line graph of the Hasse diagram of $P.$ That is,  $\mathcal{J}(P) \equiv  \mathcal{J}(L(P))$. 
 \end{corollary}
For an example of a tree-like poset, the acyclic multiplihedron for the three-edge claw poset with vertex labels is shown in Figure~\ref{fig:clawmulti}.  

\begin{figure}
    \centering
    \includegraphics[width=\linewidth]{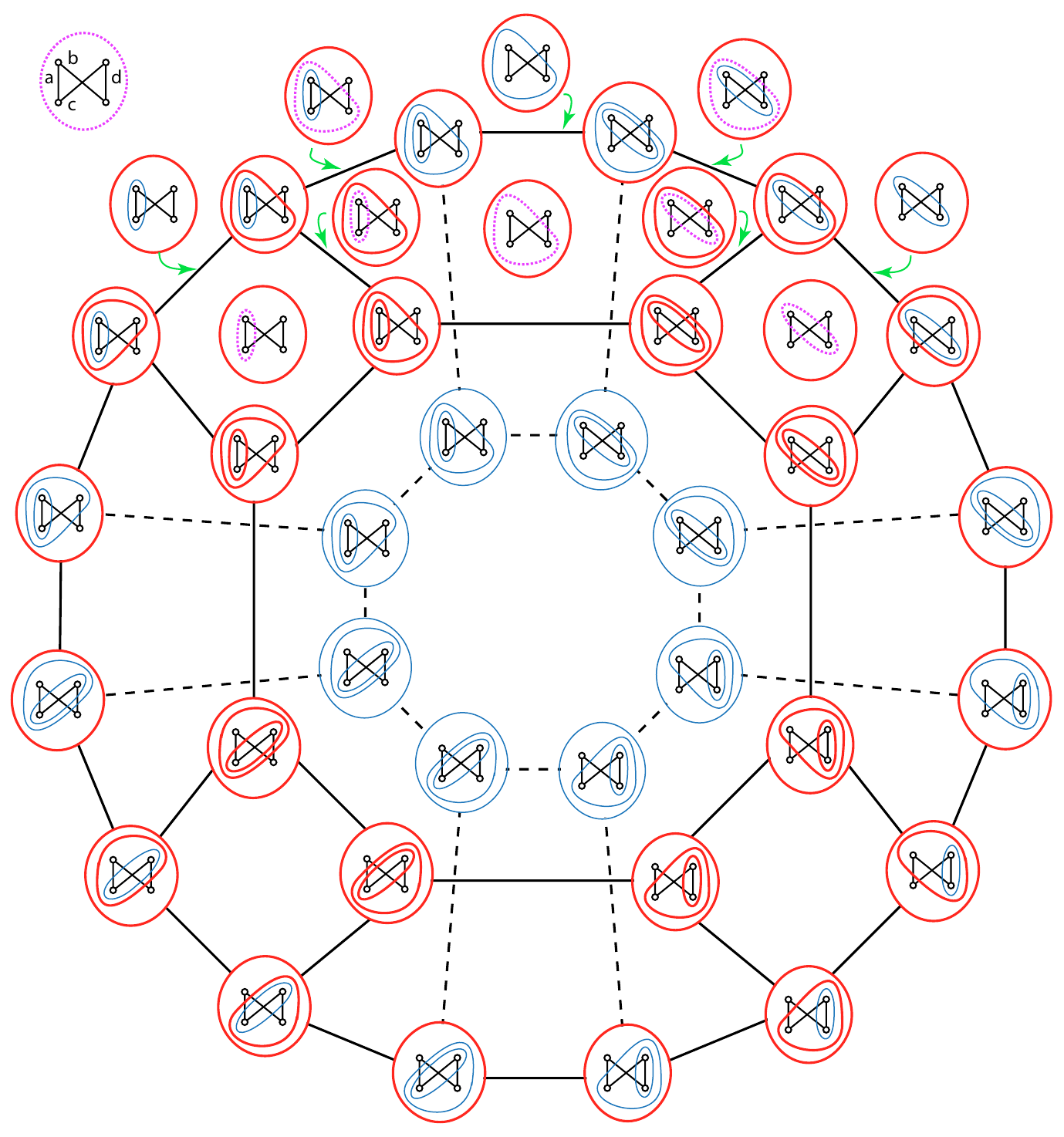}
    \caption{ Multiplihedron for bowtie poset, showing all the vertices and some additional face labels.  The $f$-vector is $(32, 48, 18).$}
    \label{fig:bowtiemult}
\end{figure}

\subsection{Realization}\label{real}

To realize the acyclic poset multiplihedron $\mathcal{J}(P)$  as a convex hull we will start with the same set of facet inequalities as for the graph multiplihedron of the line graph of the Hasse diagram, $\mathcal{J}(L(P))$. To those we add an equality for each cycle in the Hasse diagram: choosing a direction to follow that cycle we add the coordinates whose covering relation agrees with that direction and subtract the ones which do not, setting the total equal to zero.

Looking ahead, the proof that these inequalities and equalities together give a realization of the acyclic multiplihedron will rely on the fact that only faces of the graph multiplihedron corresponding to induced acyclic tubings will intersect with the hyperplane defined by the cycle equalities. Therefore we note that there is a practical choice between efficiency in constructing the realizations for posets. One
can either include an inequality for every upper and lower tubing of the line graph $L(P)$, or just for those tubings
which use only tubes that are convex in terms of the poset. The latter is faster for constructing a single
example, but if you have several posets with the same underlying line graph of all their Hasse diagrams, you
can keep all the same inequalities and change only the cycle equalities to quickly realize them all.
\begin{defn}\label{really}
    For a poset $P$ then we define the following polytope, which we will show in Section~\ref{proofy} to be in fact combinatorially equivalent to $\mathcal{J}(P)$. Here each node $v_i$ of the line graph $L(P)$ corresponds to an edge of the Hasse diagram of $P.$
\end{defn}
\begin{enumerate}
    \item For each colored piping of $P$ with only one blue pipe $t$, a lower facet given by $$\sum_{v_i\in t} x_i \ge 3^{|t|-1} .$$
    \item For each colored piping of $P$ with only one red pipe $P$ itself, and (possibly) $k$ purple pipes $t_1\dots t_k$, an upper facet given by: $$\sum_{v_i\notin \cup t_j}x_i \le 3^n-\sum_{j=1}^k3^{|t_j|}.$$
    \item For each cycle $c$ in the underlying graph of the Hasse diagram, an equality. Choosing an arbitrary direction to traverse the cycle, the orientation (covering relation) of the edge $v_i$ in $c$ will either agree with or be against the direction of traversal of $c$.
    $$\sum_{v_i \text{ with }c }x_i = \sum_{v_i \text{ against } c}x_i $$ Following \cite{mppfull} we will refer to the intersection of all the cycle equalities as the \textit{evaluation space} for the poset. Note that including an equality for every cycle is often redundant, since there are bases of cycles that are sufficient to determine the evaluation space; we will point this out in the proof.
\end{enumerate}

\begin{table}[h!]
    \centering
    \begin{tabular}{| c | c | c |}
\hline
 lower facets & upper facets & unused facets \\ \hline
    $x_a \ge 1$  &  $x_a+x_b \le 72$  &  $x_a\le 54$ \\
 $x_b \ge 1$   &  $x_c+x_d \le 72$  &  $x_b\le 54$ \\
  $x_c \ge 1$ &   $x_a+x_d \le 75$  &  $x_c\le 54$ \\
   $x_d \ge 1$    &    $x_b+x_c \le 75$  &  $x_d\le 54$ \\
   $x_a+x_b \ge 3$    &  $x_a+x_b+x_c \le 78$ &  $x_a +x_c\le 72$ \\
   $x_c+x_d \ge 3$  &   $x_b+x_c+x_d \le 78$  &  $x_b +x_d\le 72$ \\
$x_a+x_b+x_c+x_d \ge 27$  & $x_a+x_b+x_d \le 78$ & $x_a+x_b+x_c \ge 9$ \\ 
    & $x_a+x_c+x_d \le 78$ & $x_a+x_b+x_d \ge 9$ \\
  &  $x_a+x_b+x_c+x_d \le 81$ & $x_a+x_c+x_d \ge 9$ \\
\underline{~~~~~~~cycle equality~~~~~~~} & &  $x_b+x_c+x_d \ge 9$\\
$x_a-x_b+x_c-x_d = 0$ && $x_a+x_c \ge 3$\\
&& $x_b+x_d \ge 3$\\
\hline
\end{tabular}
    \caption{The realization for the diamond poset acyclic multiplihedron, using variables $(x_a,x_b,x_c,x_d)$ with subscripts from the edges of the diagram in Figure~\ref{fig:diamondmulti}. Facet labels are shown in Figure~\ref{fig:multiposetdifac}. An image of the realization is in Figure~\ref{fig:diamon_mult}.}
    \label{tab:diam}
\end{table}

\begin{figure}
    \centering
    \includegraphics[width=\linewidth]{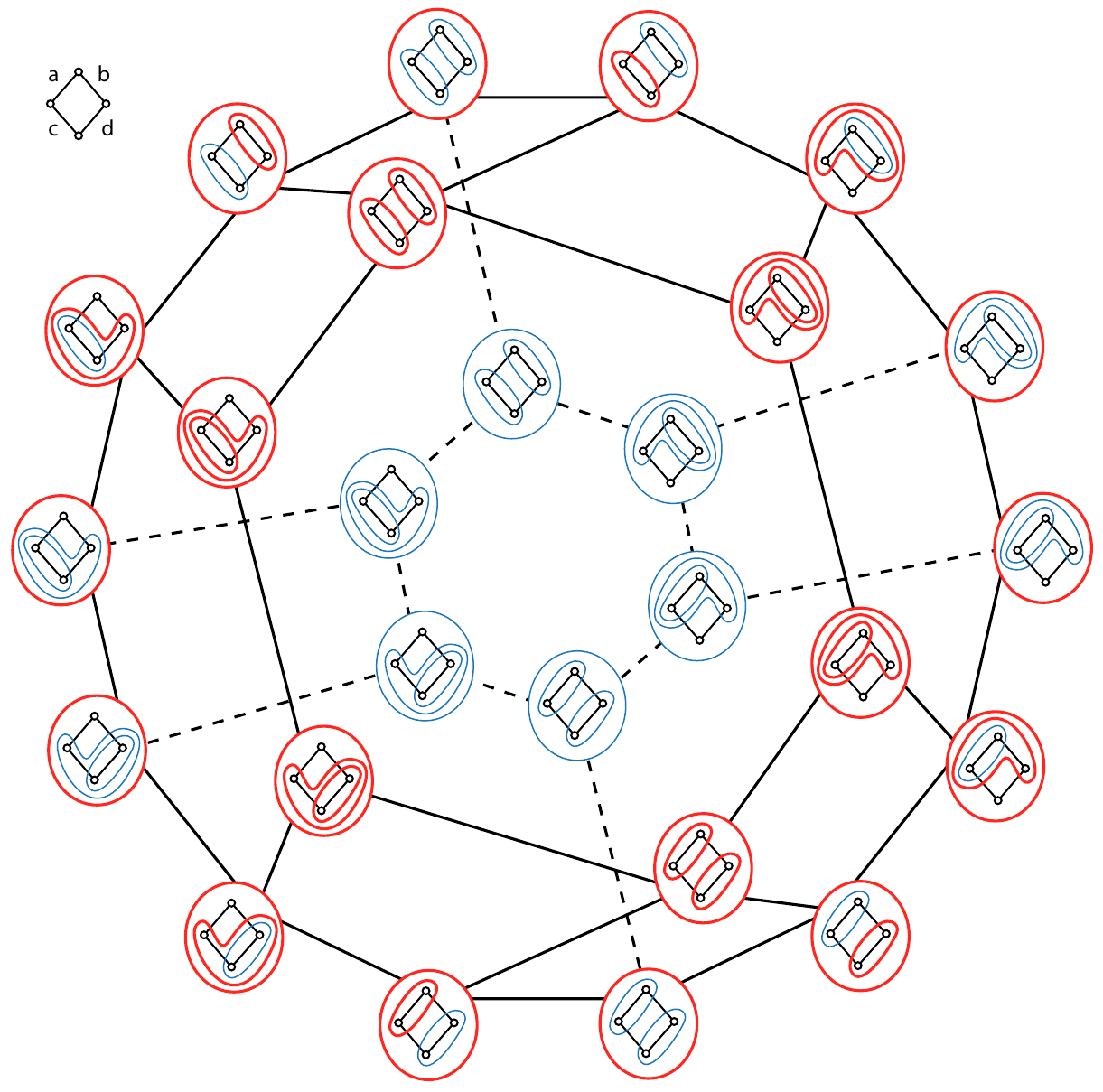}
    \caption{Multiplihedron for diamond poset. The $f$-vector is $(26, 40, 16).$ }
    \label{fig:diamondmulti}
\end{figure}

\begin{figure}
    \centering
    \includegraphics[width=\linewidth]{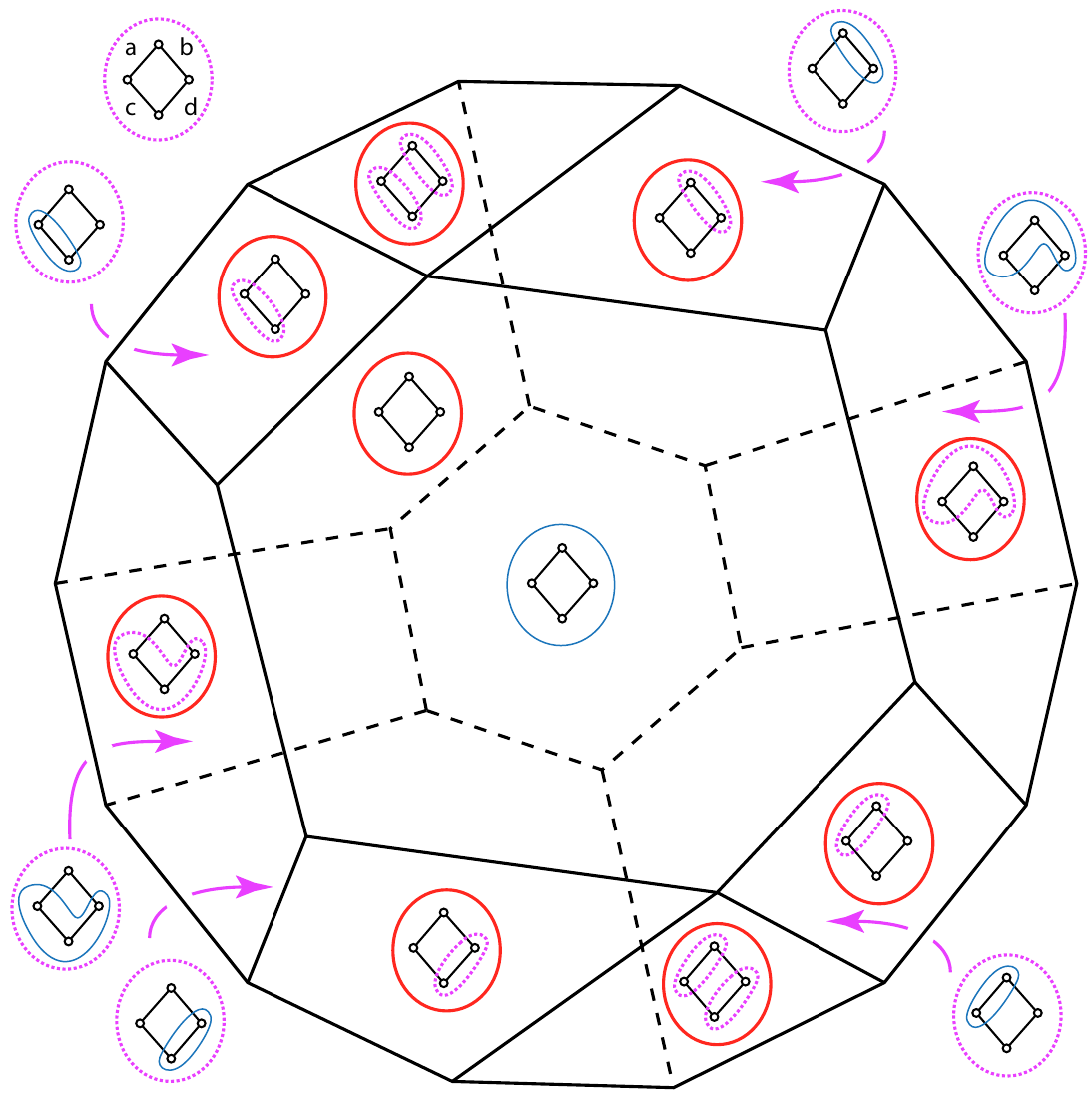}
    \caption{Multiplihedron for diamond poset, showing facet labels for the calculations in Table~\ref{tab:diam}. }
    \label{fig:multiposetdifac}
\end{figure}

\begin{figure}
    \centering
    \includegraphics[width=\linewidth]{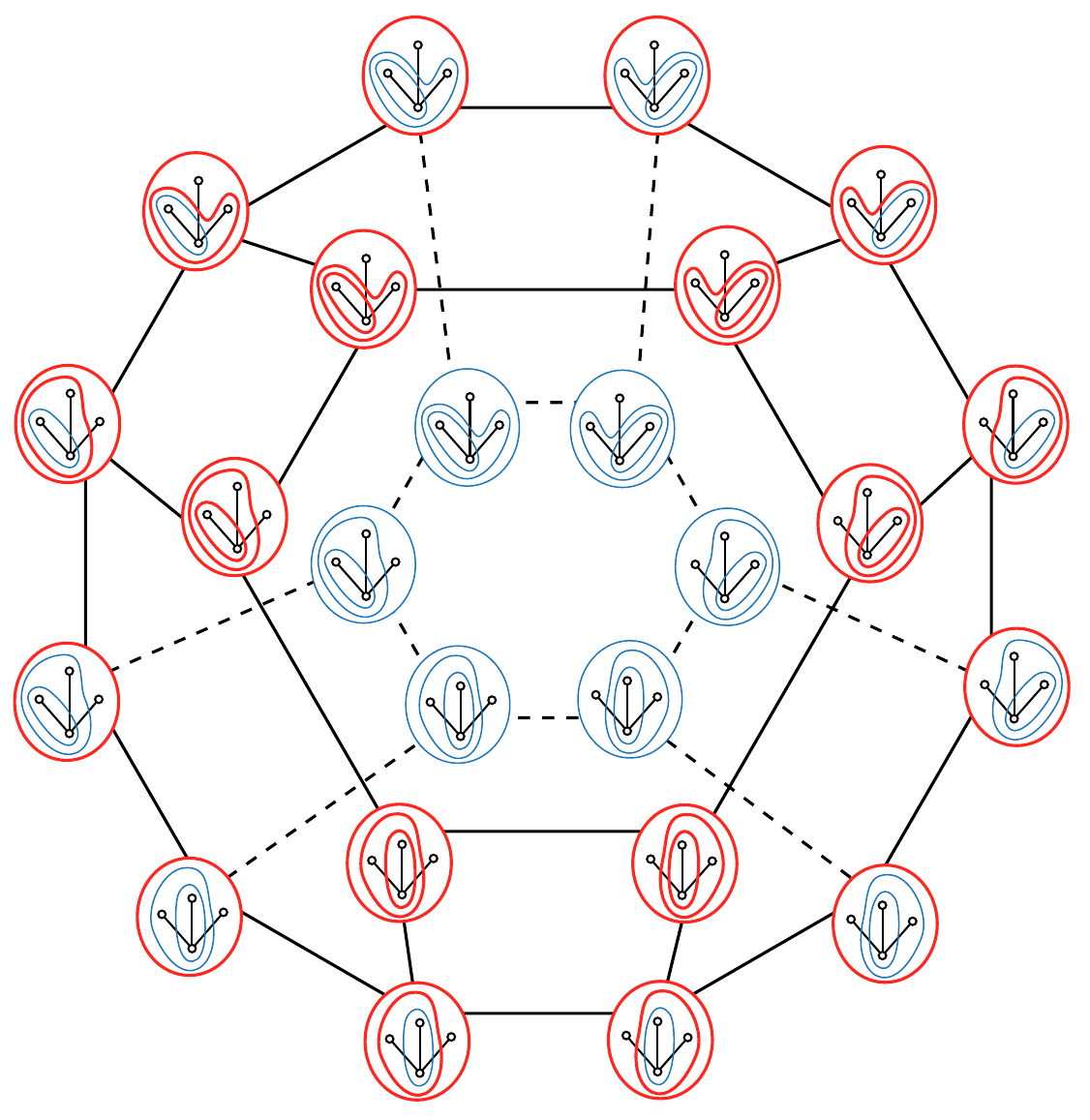}
    \caption{Multiplihedron for claw poset.  The $f$-vector is $(24, 36, 14).$ Since the Hasse diagram is a tree, this is precisely the multiplihedron for its line graph, the complete graph on three vertices. Thus it is a 3D permutohedron, as shown in \cite{dev-forc}.}
    \label{fig:clawmulti}
\end{figure}

\textbf{Examples:~} Table~\ref{tab:diam} shows the realization for the diamond poset $P$, using variables $(x_a,x_b,x_c,x_d)$ with subscripts from the edges of the diagram in Figures~\ref{fig:diamondmulti} and~\ref{fig:multiposetdifac}. The upper facets are in the foreground of those figures, and the lower facets in the background. In Table~\ref{tab:diam} we also list the unused facets of the full 4D cycle-graph multiplihedron $\mathcal{J}(L(P))$ for comparison. The choice of positive and negative for the coefficients in the cycle equality is  $x_a-x_b+x_c-x_d =0$ (the choice is immaterial since the cycle equality can be multiplied by -1.)




Table~\ref{tab:bow} shows the realization for the bowtie poset, using variables $(x_a,x_b,x_c,x_d)$ with subscripts from the edges of the diagram in Figure~\ref{fig:bowtiemult}. The upper facets are in the foreground, and the lower facets in the background. We again list the unused facets of the full 4D cycle-graph multiplihedron for comparison: notice how the two realizations in Tables~\ref{tab:diam} and~\ref{tab:bow} can both include the unused facets without affecting the outcome, and thus differ only by the cycle equalities.  \\

\begin{table}[h!]
    \centering
    \begin{tabular}{| c | c | c |}
\hline
 lower facets & upper facets & unused facets \\ \hline
    $x_a \ge 1$  &  $x_a+x_b \le 72$  &  $x_a\le 54$ \\
 $x_b \ge 1$   &  $x_c+x_d \le 72$  &  $x_b\le 54$ \\
  $x_c \ge 1$ &  $x_a +x_c\le 72$   &  $x_c\le 54$ \\
   $x_d \ge 1$    &   $x_b +x_d\le 72$   &  $x_d\le 54$ \\
   $x_a+x_b \ge 3$    &  $x_a+x_b+x_c \le 78$ & $x_a+x_d \le 75$  \\
   $x_c+x_d \ge 3$  &   $x_b+x_c+x_d \le 78$  &  $x_b+x_c \le 75$ \\
$x_a+x_c \ge 3$  & $x_a+x_b+x_d \le 78$ & $x_a+x_b+x_c \ge 9$ \\ 
   $x_b+x_d \ge 3$ & $x_a+x_c+x_d \le 78$ & $x_a+x_b+x_d \ge 9$ \\
 $x_a+x_b+x_c+x_d \ge 27$ &  $x_a+x_b+x_c+x_d \le 81$ & $x_a+x_c+x_d \ge 9$ \\
 & &  $x_b+x_c+x_d \ge 9$\\
\underline{~~~~~~~cycle equality~~~~~~~} && \\
$x_a-x_b-x_c+x_d = 0$&& \\
\hline
\end{tabular}
    \caption{The realization for the bowtie poset acyclic multiplihedron, using variables $(x_a,x_b,x_c,x_d)$ with subscripts from the edges of the diagram in Figure~\ref{fig:bowtiemult}. An image of the realization is in Figure~\ref{fig:diamon_mult}.}
    \label{tab:bow}
\end{table}












\begin{figure}
    \centering
    \includegraphics[width=.25\linewidth]{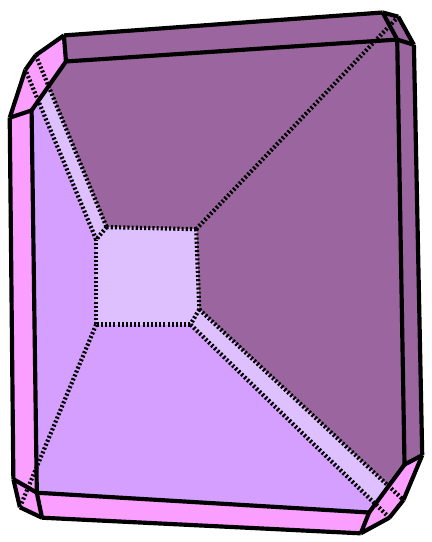}\,\,\,\,\,\,\,\,\,\,\,\,\,\,\,\,\,\,\,\,\,\,\,\,\,\,\,\,\,\,\,\,\,\,\,
     \includegraphics[width=.265\linewidth]{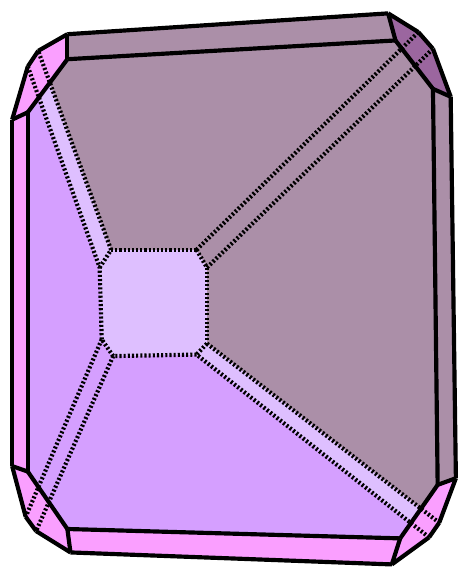}
    \caption{ On the left: acyclic multiplihedron realization for the diamond poset, as seen in Figure~\ref{fig:diamondmulti}. On the right: acyclic multiplihedron realization for the bowtie poset, as seen in Figure~\ref{fig:bowtiemult}.}
    \label{fig:diamon_mult}
\end{figure}

\subsection{Proof of the realization}\label{proofy}
We will show that the combinatorial description of the acyclic poset multiplihedra is realized by the geometric realization we give above in Definition~\ref{really}. 

The proof follows the same logic as in \cite{mppfull}, in which the authors show that the poset acyclonestohedra themselves  are found as cross sections of geometric realizations of graph associahedra. Given a poset $P$ with its Hasse diagram, they take the line graph $L(P)$ of that Hasse diagram and a realization of the graph associahedron of that line graph. Then they add equalities for each cycle in the original Hasse diagram; together the intersection of these equalities defines what they call the evaluation (sub)space. The intersection of this evaluation space with the graph associahedron $\mathcal{K}(L(P))$ is shown to be the acyclonestohedron. The proof proceeds via oriented matroid theory, and shows that a face of the graph associahedron intersects all the hyperplanes in the evaluation space if and only if that face corresponds to a tubing which arises from an acyclic tubing (piping) of the Hasse diagram. This demonstrates the desired geometric realization simply because the structure of the acyclonestohedron is precisely the structure of the graph associahedron (of the line graph of the Hasse diagram), restricted entirely to the acyclic faces.

Our proof will follow that logic exactly: starting with a realization of the graph multiplihedra. Thus the proof requires two lemmas: first that any face which corresponds to a not-acyclic tubing (so cyclic!) has its vertices all on one side of the evaluation space, so that the face itself does not intersect the evaluation space at all. The proof of this lemma will be checking that our vertex coordinates obey the correct inequalities: if not acyclic then all associated vertex evaluations will be all greater than (or all less than) at least one of the cycle equalities.

The second lemma is that any face which corresponds to an acyclic tubing does in fact intersect the evaluation subspace. For that to be true, we need such a face to have its set of vertices distributed around the evaluation space, to ensure that the face is intersected by that space. 
Specifically, the set of vertices must include points that reside in each chamber of the central hyperplane arrangement whose common intersection is the evaluation space. In terms of the oriented matroid, the chambers are represented by sign vectors using the symbols $\{+,-\}$ to refer to the position relative to each of $r$ oriented hyperplanes.

This we will demonstrate just by considering an acyclic face, and showing how to construct its vertices that lie in each chamber.
The same basic fact supports both our lemmas: if a tubing of the line graph is acyclic (induced by an acyclic piping), it must be relatively small, with several different (usually cyclic) ways to extend it to a maximal (vertex) tubing of the multiplihedron. Conversely, if the tubing is already cyclic, it is large enough that the extensions are all similar, in fact all on one side of a hyperplane. 

\begin{lemma}\label{uno}
    Let $F_T$ be a face of the graph multiplihedron $\mathcal{J}(L(P))$ associated to a non-acyclic tubing $T$ of the line graph of the Hasse diagram of a poset  $P$. Then for some cycle $c$ of the Hasse diagram of $P$ all the vertices in $F_T$ lie on one side, in the same halfspace, of the hyperplane associated to that cycle. 
    \end{lemma}
    \begin{proof}
         All the vertices of $F_T$ will be extensions of $T$ via adding red or blue tubes and resolving any purple tubes. Let $c$ be a cycle of the Hasse diagram on which $T$ exhibits its non-acyclic character.  Then the non-acyclic tubing $T$ must have either a tube $t$ which is not convex (involving $c$, so containing part of $c$ but failing convexity via $c$), or a set of non-nested tubes which is not acyclic, again involving $c$. In the first case,  $t$ is a single pipe which  contains all the edges of $c$ whose orientation in the Hasse diagram agrees with a particular traversal of $c$ (up or down in the Hasse diagram, depending on the choice of clockwise or counterclockwise) but not all the edges of $c$ which go against that traversal. Recalling that edges are nodes of the line graph, that means there are some nodes outside of $t$ but part of $c$.
         
         The cycle equality from  Definition~\ref{really} for cycle $c$ is as follows: $$\sum_{v_i \text{ with }c }x_i = \sum_{v_i \text{ against } c}x_i.$$ Now we consider evaluating the left hand side of that cycle equality for $c$.  Choosing the traversal of $c$ to be in the direction of the edges contained in $t$ we see the chain of inequalities: 
         

$$\sum_{v_i\text{ with }c} x_i
\le \sum_{v_i \in t} x_i
< \sum_{v_i\in (c - t)} x_i
\le \sum_{v_i\text{ against }c} x_i .$$

The second, strict, inequality above is  due to the formulas, from Section~\ref{vert}, for $x_i$ in a vertex (maximal) tubing. The value of $x_i$ grows exponentially (base 3) with the number of nodes (in smaller tubes) inside $\epsilon(v_i)$,  the smallest tube that contains node $v_i$.   For instance,  for a blue tube $t$ the values of the $x_i$ for nodes inside it sum to $3^{|t|-1}.$ By the assumptions on $t$ and $c$ there exists some $x_j$ in $c$, outside of but adjacent to $t$, which has value at least $3^{|t|}-3^{|t|-1}$, larger than the sum inside $t.$ (That follows from the basic inequality: $ 3^k+3^m < 3^{k+m+1},$ for $k,m$ non-negative integers.)\footnote{Notice that strict inequality does not hold for base 2, since $2^0 +2^0  = 2^1.$} Thus the inequality holds for blue tubings; for multicolored tubings the colors involved preserve the inequality since red tubes (which contribute an extra factor of 3) can only be inside other red tubes. Thus the outside-adjacent node value $x_j$ will increase by a factor of three if $\epsilon(v_j)$ is red, while the sum of the inside nodes increases by at most a factor of 3. 
Therefore the cycle equality fails, strictly, so  all vertices of the face of the tubing $T$ must lie on the same side of the cycle hyperplane for $c$, so  $F_T$ does not intersect the evaluation space.

The inequalities are similar for the case in which $T$ has a set of non-nested tubes $t_1,\dots,t_k$ whose union is non-acylic with respect to $c$.  Here the union $\bigcup t_j$ of non-nested tubes contains all edges of $c$ whose orientation in the Hasse diagram agrees with a particular traversal of $c$ but not all the edges of $c$ which go against that traversal. Recalling that edges are nodes of the line graph, that means there are some nodes outside of the union but part of $c$. Choosing the traversal of $c$ to be in the direction of the edges contained in the union we see the chain of inequalities:  


$$\sum_{v_i\text{ with }c} x_i 
\le \sum_{j=1}^k\sum_{v_i \in t_j} x_i 
< \sum_{v_i\in (c - \bigcup t_j)} x_i
\le \sum_{v_i\text{ against }c} x_i .$$

This time the second, strict inequality is seen via  several nodes of $c$: there exist a set of such nodes with coordinates $x_j$ in bijection with the tubes $t_j$, each $x_j$ for a node outside but adjacent to $t_j.$ Again for all blue tubes, the sum of the coordinates inside tube $t_j$ is $3^{|t_j|-1}$ while the value of $x_j$ is greater than $3^{|t_j|} - \sum 3^{|t_j|-1}.$ Thus the sum of the outside nodes is always a larger value than the sum inside. Again the colors of the tubes involved do not change that inequality, since red tubes (which contribute an extra factor of 3) can only be inside other red tubes. 

\textbf{Example}: Figure~\ref{matroid} shows a pair of pipings that demonstrates the inequalities for Lemma~\ref{uno}. Next to each non-acyclic piping $T$ is a maximal tubing that corresponds to a graph multiplihedron vertex in the facet $F_T.$ On the top row we have $x_1+x_2+x_3+x_7+x_8 = 3^4$ and $x_4 = 3(3^6-3^4-3^0)$, and the chain of inequalities is: 
$$x_1+x_3+x_7+x_8\le x_1+x_2+x_3+x_7+x_8 < x_4+x_5+x_6 \le x_2+x_4+x_5+x_6. $$

On the bottom row we have $x_1+x_3+x_7+x_8 = 1+1+1+3(3^2-2)$, $x_2=3(3^6-3^2-3^2)$, and $x_4 = 3^2-2$, and the chain of inequalities is:
$$x_1+x_3+x_7+x_8\le x_1+x_3+x_7+x_8 < x_2+x_4+x_5+x_6 \le x_2+x_4+x_5+x_6. $$
\begin{figure}
    \centering
    \includegraphics[width=0.75\linewidth]{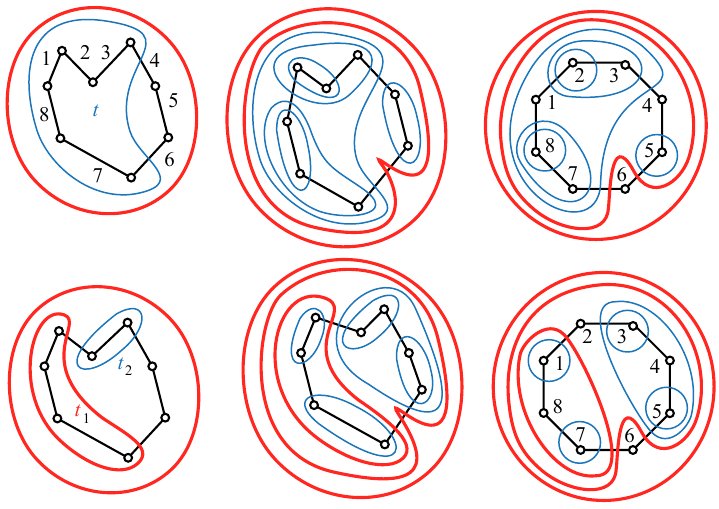}
    \caption{Demonstrating a non-convex pipe $t$ (top row) and a non-acyclic piping with $t_1,t_2$ (bottom row). The tubings (in the middle and right) are vertex tubings of the graph multiplihedron, first drawn as pipings of the edges of the poset, then as tubings of the line graph.}
    \label{matroid}
\end{figure}

        \end{proof}    

        \begin{figure}[h!]
    \centering
    \includegraphics[width=\linewidth]{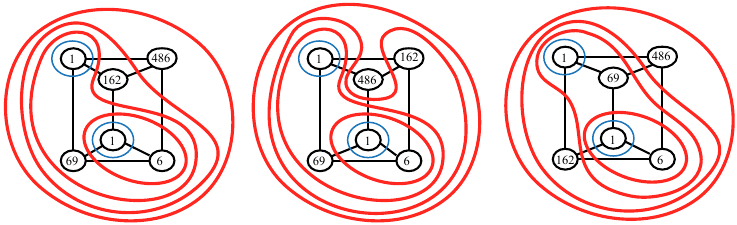}
    \caption{Examples of values assigned to the variables associated to each node in some maximal colored tubings, in the realization of the graph multiplihedron for the graph shown. These are specifically three ways (out of six ) to add tubes to the maximal acyclic poset tubing in Figure~\ref{fig:runonecorre} to get a maximal graph tubing.}
    \label{fig:runonenum}
\end{figure}
 \begin{lemma}
     Let $F_T$ be a face of the graph multiplihedron $\mathcal{J}(L(P))$ associated to an acyclic tubing $T$ of the line graph of the Hasse diagram of $P$. Then $F_T$ intersects the evaluation space, which is the intersection of all the cycle hyperplanes.
 \end{lemma}       
        
   \begin{proof}
       We consider an arrangement of cycle hyperplanes which determines the evaluation space. This arrangement is not unique, but we can choose one of the smallest such arrangements that determines the unique evaluation space.  
       The number of cycle hyperplanes is the same as the number of independent cycles in the (underlying graph of the) Hasse diagram, the cycle rank: $r = |E| - n +1$ for a connected diagram of $n$ nodes. (Kirchoff's cyclomatic number.) Notice that combining the equations of cycles gives a new cycle equation. We choose independent cycles by selecting a minimum set of edges in the Hasse diagram that need to be removed to make the underlying graph a tree, and then match a unique cycle to each of those edges, so that each cycle is broken by its matched edge. This choice provides a cycle basis, and the common intersection of the hyperplanes is the evaluation space. 


Thus the cycle equalities we have chosen are linearly independent and there are fewer of them than the dimension we are working in, which is |E|, the number of edges in the Hasse diagram. Therefore the number of chambers is exactly $2^r$,  where $r =|E|-n+1$.  These chambers have all possible sign vectors of $r$ components from $\{+,-\}$. 
The number of red and blue pipes in $T$ is at most $n-1$, including the universal pipe, since we are considering an acyclic tubing of an $n$-node poset. To find any vertex of our face $F_T$ in the graph multiplihedron, combinatorially, we have to add tubes to $T$ until we get the max number of red and blue tubes, which is $|E|$. 
So in the case where we have a maximal acyclic poset tubing $T$, we need exactly $r$ more red and blue tubes to get $|E|$ of them. 

We claim that the final calculation of the coordinates from the maximal tubings will achieve all the combinations of $>$ and $<$ for the cycle equalities, reaching all $2^r$ chambers. To demonstrate our claim we need to add  tubes so that each choice of tube determines a $+$, or  $-$ , that is, the point obeys an inequality $>$ or $<$  for each of the cycle hyperplanes.   In the case where we have fewer than $n-1$ tubes in $T$, we start by arbitrarily adding more (and resolve some purple tubes) to get a maximal acyclic tubing.

 Given the maximal acyclic tubing $T$,  for each cycle $c$ of the Hasse diagram there will be a pair of edges of the Hasse diagram neither of which are the only edge of a single tube, but are either both in or both out of every tube in $T.$ Furthermore, in each pair, given a traversal of the cycle, one edge will be oriented with the direction of the traversal, and the other against. These pairs might overlap, sharing an edge between several pairs, specifically when cycles overlap. Choosing exactly one edge from each pair and deleting it produces a spanning tree; and this choice determines a cycle basis for the evaluation space. Taking that collection of cycles and considering its hyperplanes, we claim there is a vertex of $F_T$ from the graph multiplihedron in each chamber of the central arrangement. The claim follows since the collection of edges in the pairs mentioned are nodes in the line graph that can be completely ordered in nested depth via the additional tubes (in groupings of these nodes that correspond to cycles sharing an edge.) Furthermore, for each pair of edges we can choose the tubes to make either edge in the pair at a greater nested depth, and that is a decision independent of the options chosen for the other pairs.  Then for any hyperplane of the chosen cycle basis, we can (independently of the other hyperplanes) force the vertex to lie on either side of it, due to the exponential growth of the coordinate values as nested depth decreases. Examples follow; note that the color of the added tubes again preserves the inequalities needed since the blue tubes cannot be outside of red, and the red tubes confer an extra factor of 3. 
   \end{proof}


       \textbf{Examples}: Figure~\ref{fig:runonecorre} shows a maximal acyclic piping on  $P$  the double claw poset with six covering relations shown. The same piping is shown as a tubing of the line graph of the Hasse diagram, and in Figure~\ref{fig:runonenum} are shown three of six vertex tubings of the graph multiplihedron found by adding tubes to that tubing. The leftmost can be checked to obey the inequalities: $x_a+x_d < x_b + x_e < x_c+x_f$ which  lies in the chamber designated $<-,->,$ using the two hyperplanes  $x_a+x_d = x_b + x_e$ and $x_b + x_e = x_c+x_f.$ The center obeys $x_a+x_d < x_b + x_e > x_c+x_f$ which  lies in the chamber designated $<-,+>.$ The right obeys $x_a+x_d > x_b + x_e < x_c+x_f$ which  lies in the chamber designated $<+,->.$ The fourth maximal tubing needed to lie in the chamber $<+,+>$ is left to the reader.
       A larger example is shown in Figure~\ref{fig:bigone}, to see the pairs of edges that can occur in a maximal acyclic piping with several separate cycles.

       \begin{figure}
    \centering
    \includegraphics[width=0.5\linewidth]{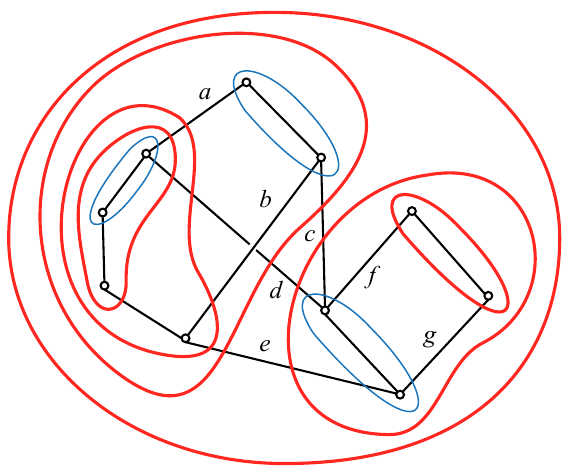}
    \caption{A maximal acyclic piping of a poset with 4 independent cycles. These can be chosen so that $\{a,b\}, ~\{c,d\}, ~\{d,e\},$ and $\{f,g\}$ are the pairs of edges in each independent cycle such that the pair is either both outside or both inside all pipes. Choosing one edge from each pair to delete gives a spanning tree. Adding tubes to make a maximal tubing of the line graph can be done so as to put the resulting points in each of the chambers of the hyperplane arrangement.}
    \label{fig:bigone}
\end{figure}



       


\subsection{New realization for acyclonestohedra}

Since the multiplihedron has two facets that are copies of the associahedron, one for all red tubes and one for all blue tubes, we can realize the acyclic poset associahedra by using the inequalities just for one of these, the cycle equalities, and an equality for the entire facet. For the all blue facet, that would be just the lower facet inequalities, with equality for the universal blue tube, $\sum_{x_i} = 3^{n-1}$ for $n$ edges in the poset.

Table~\ref{tab:newby1} shows the realization for the acyclic poset associahedron shown in Figure~\ref{fig:newby1}. The image of this realization is pictured in Figure~\ref{fig:newby1real}.

\begin{table}[h!]
    \centering
    \begin{tabular}{| c | c | c |}
\hline
 (lower) facets &  & equalities \\ \hline
    $x_a \ge 1$  &  $x_c+x_d \ge 3 $  &  \\
 $x_b \ge 1$   &  $x_d+x_e \ge 3 $  & \underline{~~~~~~~facet equality~~~~~~~}  \\
  $x_c \ge 1$ &  $x_a+x_e \ge 3 $   &  $x_a+x_b+x_c+x_d+x_e = 81$ \\
   $x_d \ge 1$    &  $x_a+x_b+x_c \ge 9$    &   \\
   $x_e \ge 1$    & $x_c+x_d+x_e \ge 9$   & \underline{~~~~~~~cycle equality~~~~~~~} \\
   $x_a+x_b \ge 3$    &  $x_a+x_d+x_e \ge 9$  & $x_a-x_b+x_c-x_d+x_e = 0$  \\
   $x_b+x_c \ge 3$  &   $x_a+x_b+x_e \ge 9$   &   \\
\hline
\end{tabular}
    \caption{The realization for the fishy poset acyclic associahedron, using variables $(x_a,x_b,x_c,x_d,x_e)$ with subscripts from the edges of the diagram in Figure~\ref{fig:newby1}. An image of the realization is in Figure~\ref{fig:newby1real}.}
    \label{tab:newby1}
\end{table}

     \begin{figure}
         \centering
         \includegraphics[width=\linewidth]{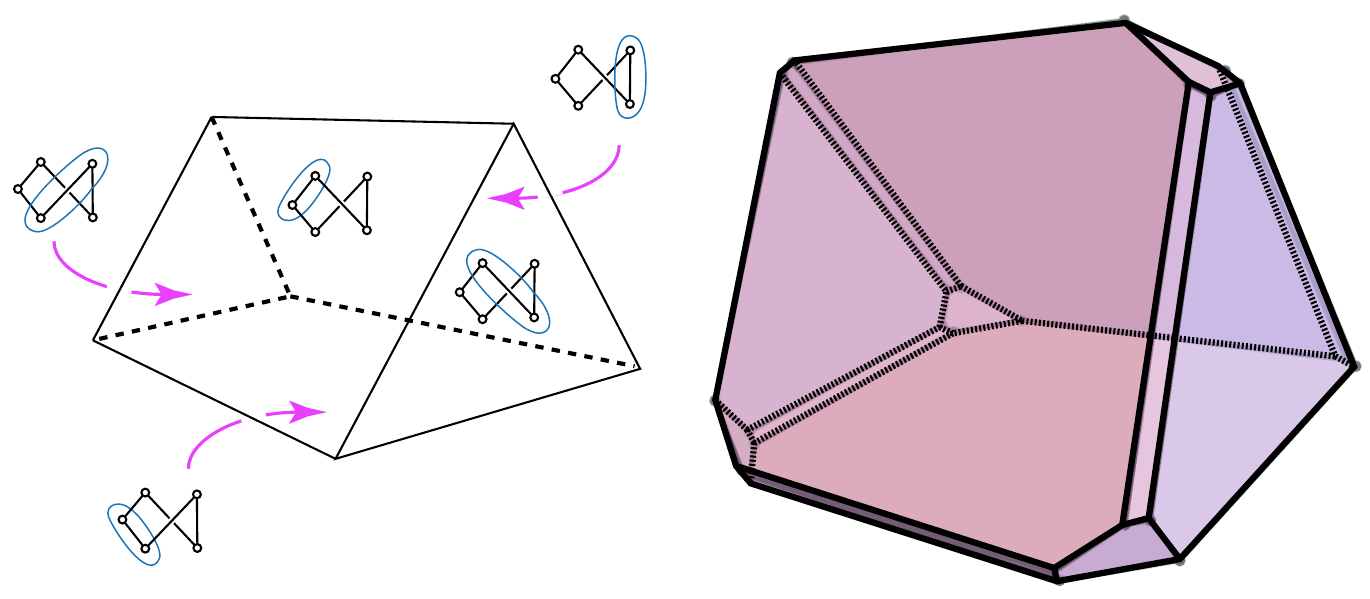}
         \caption{On the left is the order polytope $\mathcal{O}(P)$ for the poset from Figure~\ref{fig:newby1}. On the right is an image of our new realization for the acyclonestohedron.}
         \label{fig:newby1real}
     \end{figure}

\begin{figure}
    \centering
    \includegraphics[width=\linewidth]{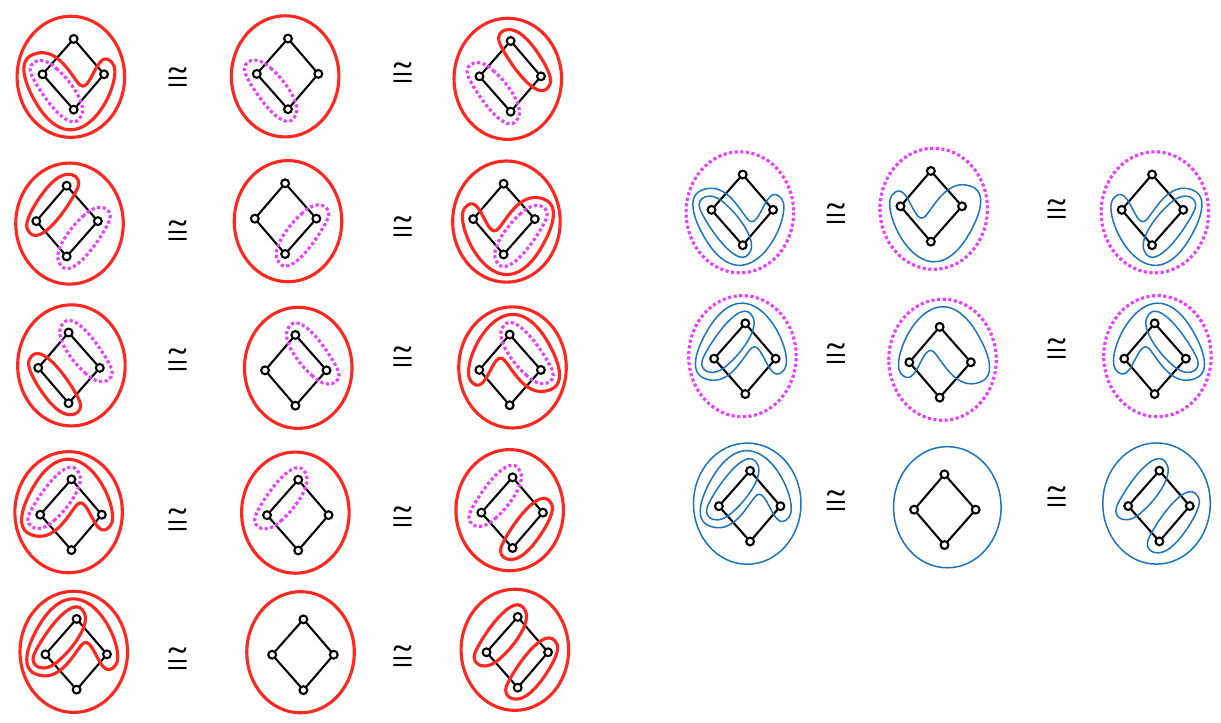}
    \caption{Example equivalences involving facets of the acyclic (diamond poset) multiplihedron which become lower dimensional faces of the quotients, under addition of red or blue tubes. These are facet pipings (center of each column of equivalences) corresponding to facet inequalities to be dropped from the multiplihedron realization to achieve (left column) the acyclic design associahedron (cubeahedron) and   (right column) the acyclic composihedron.}
    \label{fig:equivd}
\end{figure}

\begin{figure}
    \centering
    \includegraphics[width=\linewidth]{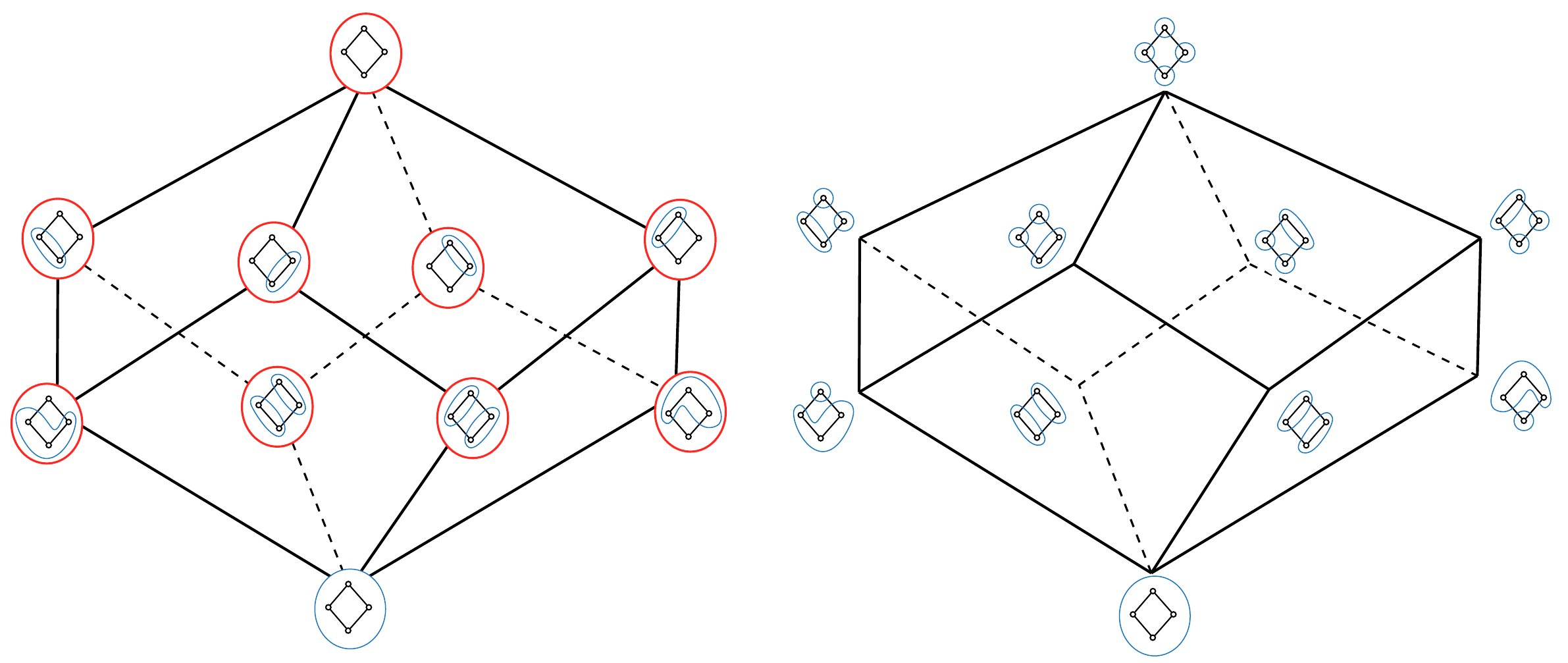}
    \caption{Vertex labels of the quotient of the diamond poset acyclic multiplihedron, under both equivalences: seen in bijection with the vertex labels of the interval polytope of the order polytope of the same poset. These vertices are the faces of the order polytope, as seen in Figure~\ref{fig:trunc1}. }
    \label{fig:intermap}
\end{figure}

\begin{figure}
    \centering
    \includegraphics[width=\linewidth]{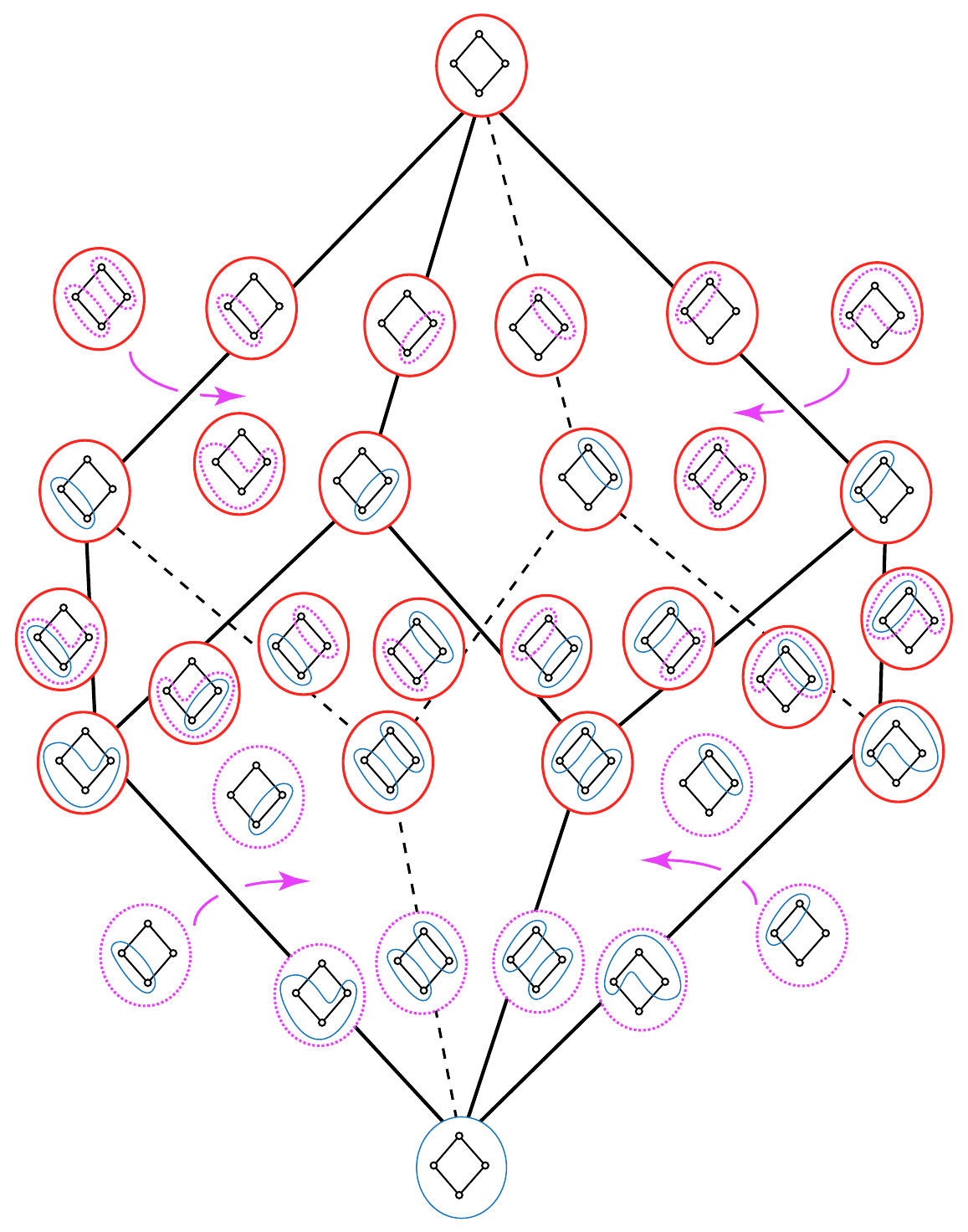}
    \caption{All face labels of the quotient of the diamond poset acyclic multiplihedron, under both equivalences.}
    \label{fig:intervalsj}
\end{figure}

\section{Quotients of the multiplihedra}\label{conjy}

In the original description of the multiplihedron, Stasheff  considered a simplification for the case of a function (a homotopy homeomorphism) from a homotopy associative space, to a strictly associative space. The multiplihedron in this case collapses to become the associahedron. The corresponding equivalence for colored tubings is generated by equating tubings which only differ by the addition of one red tube. For the graph multiplihedron it has been shown that these equivalence classes yield the polytope known as the graph cubeahedron, with faces also described by design tubings. 

In enriched category theory the alternative quotient was also seen to be important: the case in which the homotopy homomorphism has an associative domain and homotopy associative codomain. This corresponds to equivalence for colored tubings generated by equating tubings which only differ by the addition of one blue tube. These \textit{graph composihedra} are described in \cite{dev-forc} as the domain quotients of graph multiplihedra.

Finally, the equivalence relation generated by both kinds of tube addition, for graph multiplihedra, is seen to give the Boolean lattice, and more importantly, the face poset of the hypercubes.  

Here we define the acyclic poset cubeahedra and acyclic poset composihedra as the posets resulting from equivalence relations on the acyclic poset multiplihedra generated by addition  of red and blue pipes respectively.
\begin{defn}
    Given a poset $P$ define the acyclic composihedron   $\mathcal{J}_d(P)$ to be the poset of equivalence classes of acyclic pipings of the Hasse diagram, under the equivalence relation generated by addition of an acyclic blue pipe to a given piping. Define the acyclic design associahedron, or cubeahedron, $\mathcal{J}_r(P)$ to be the poset of equivalence classes of acyclic pipings of the Hasse diagram, under the equivalence relation generated by addition of a red acyclic pipe to a given piping.
\end{defn}

Taken together, the equivalence on maximal pipings (vertices of the acyclic multiplihedron) under both sorts of addition is seen to give a lattice that could be termed the acyclic poset version of the Boolean lattice, or the \textit{acyclic Boolean lattice}. However, we show here that it is already correctly described as the lattice of faces of the order polytope. More importantly, on all pipings, the equivalence under both sorts of addition of tubes is conjectured to give the face lattice of a new polytope that is actually recognizable. It is in fact the interval polytope of the face lattice of the order polytope of the poset. We can show a clear poset bijection, but the realization as a polytope is more difficult at first look.

\begin{theorem}\label{th:two}
    For a poset $P$ the equivalence classes of colored pipings under additions of red or blue pipes form a poset that is isomorphic to the poset of intervals in the face lattice of the order polytope of $P.$
\end{theorem}
\begin{proof}

  We describe the bijection $f$ which takes an interval $A$ of acyclic partitions, in the face lattice of the order polytope $\mathcal{O}(P)$, to an equivalence class $f(A)$ of the colored pipings.

   Given a collection $A$ of connected acyclic partitions of the nodes of $P,$ where $A = [X,Y]$ is an interval in the face lattice of $\mathcal{O}(P)$, we create a colored piping from $A$ by: 
  \begin{enumerate}
      \item Discard the singleton parts in any member of $A$; that is ignore them for the purpose of defining the following pipes;
      \item  For the minimal member $X$ of $A$, any non-singleton part of $X$ which is broken into smaller parts in the remainder of the interval becomes a purple pipe of  $f(A).$
      \item For the minimal member $X$ of $A$, any non-singleton part of $X$ which is not broken into smaller parts in the remainder of the interval becomes a blue pipe of  $f(A).$
      \item For the maximal member $Y$ of $A$, any non-singleton part of $Y$ becomes a blue pipe of $f(A).$
      \item If the minimal member $X$ of $A$ is not the minimal element of the lattice of $\mathcal{O}(P)$ (the trivial partition) then the universal pipe of $f(A)$ is red. 
   \end{enumerate}
      The construction of $f(A)$ is seen to give a minimal representative of the equivalence class of colored tubings, from which no red or blue tubes can be removed without changing that class.  The vertices of the interval lattice (faces of the order polytope lattice) are taken to vertices of the quotient lattice of tubings. Also, the ordering relation of the intervals is preserved by our map, since containment of sub intervals becomes resolving purple tubes (when red, the resolved tube is then dropped to reflect the equivalence under addition of red inside of red), or addition of blue tubes inside of purple  (when added outside a blue tube, the inner tube is dropped to reflect the equivalence under addition of blue inside of blue).
\end{proof}

Examples of the mapping from the proof of Theorem~\ref{th:two} can be seen in Figures~\ref{fig:intervalsj} and~\ref{fig:intermap}. The latter figure shows the vertex labels of both polytopes for the diamond poset, in corresponding location. Then  any face shown in Figure~\ref{fig:intervalsj} can be compared with the corresponding face in the right-hand polytope of Figure~\ref{fig:intermap}. The face labels of the interval polytope are simply $[X,Y]$ where $X$ and $Y$ are the minimal and maximal vertices of that face, that is, the minimal and maximal elements of an interval in the face lattice of the order polytope.

When the poset $P$ is acyclic, with a Hasse diagram that is a tree, then the order polytope $\mathcal{O}(P)$ is a simplex. The acyclic Boolean lattice is the Boolean lattice, and the interval polytope of the order polytope is the hypercube. This situation recaptures a special case of quotients of the graph multiplihedron (of the line graph of the Hasse diagram of $P$) as described in the general case in \cite{dev-forc}.

By using the realization of the acyclic poset multiplihedron, there is an obvious way to achieve realizations of the quotients. However it is not just by performing the quotient on the graph multiplihedron and then finding a cross section of that quotient polytope. That is due to the fact that the quotient on the graph multiplihedron is under additions of red and blue tubes regardless of acyclicity. Note for instance that the corresponding cross section of the hypercube does not yield the trapezohedron (the trapezohedron is the correct result for the diamond poset). Rather it is required to follow the other order, first simplify the new realization of the acyclic multiplihedron and then take the quotient. Thus the realizations for these quotients need a dedicated proof, so we conjecture here and leave that proof for future work. 
\begin{figure}
    \centering
    \includegraphics[width=\linewidth]{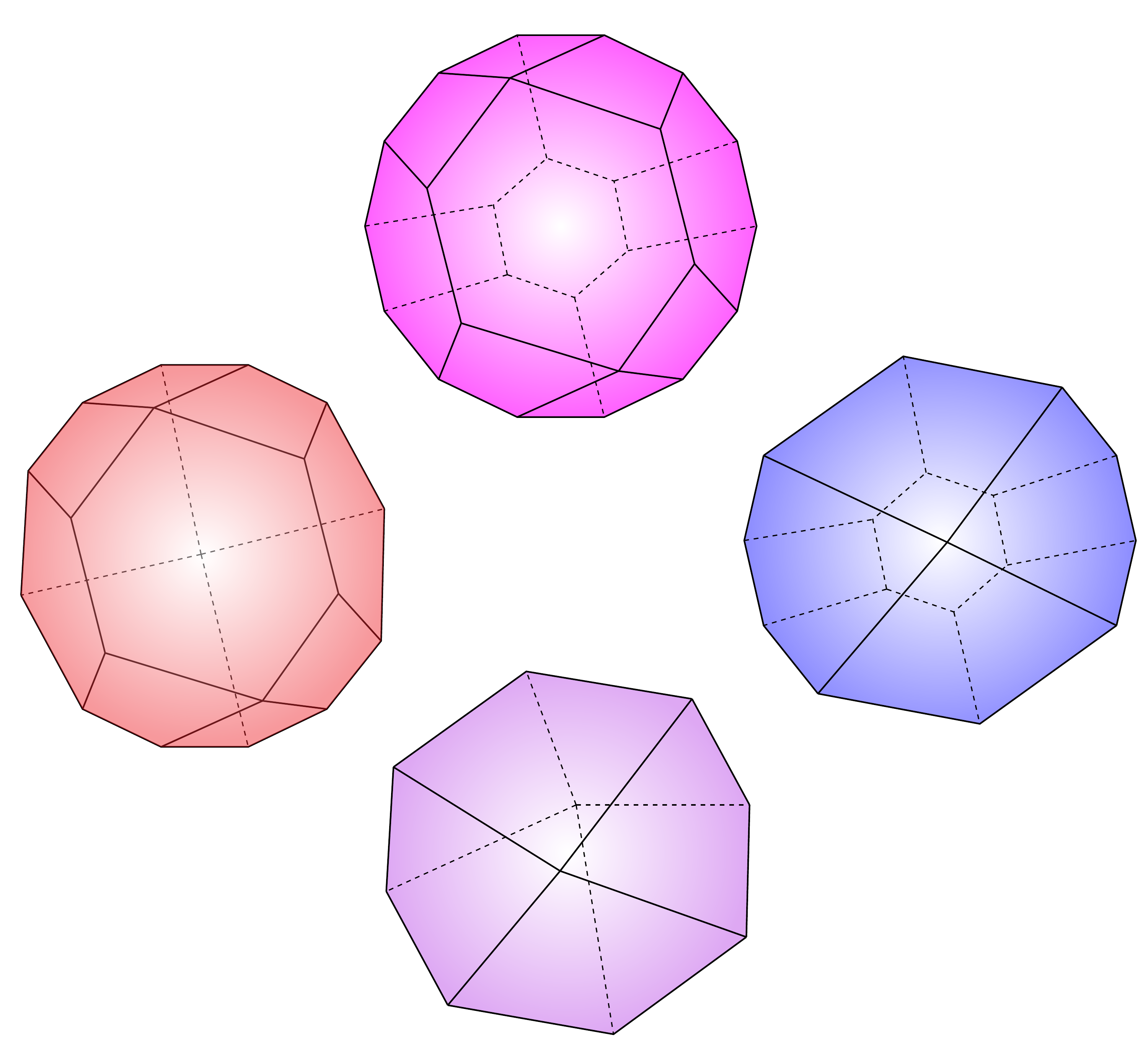}
    \caption{Top: the multiplihedron for the acyclonestohedron of the four-element diamond poset. On the left is the domain quotient (acyclic composihedron), on the right the range quotient (acyclic design associahedron, or acyclic cubeahedron), and at the bottom is the interval polytope of the order polytope of the diamond. Note that this interval polytope of the order polytope is a trapezohedron, with a side view in Figure~\ref{fig:intervalsj}. It turns out to be the same polytope as for the bowtie poset, as seen in Figure~\ref{fig:topp}. }
    \label{fig:quotientdiam}
\end{figure}

\begin{conjecture}
Given equivalence relations on colored pipings of a poset $P$, generated by 1) addition of blue pipes, 2) addition of red pipes, and 3) generated by both additions; the resulting posets are realized as polytopes found by collapsing faces of the acyclic multiplihedron. In the third case, with both equivalences, the result is the interval polytope of the order polytope of $P.$
\end{conjecture} 

To realize each quotient, experimentally it appears to give the correct result if we remove from the realization of the full acyclic poset multiplihedron the facet inequalities which correspond to facets now equivalent to lower dimensional faces by addition of a single red or blue acyclic tube. At the same time the remaining facet inequalities must be extended, with their constant set equal to the max power of 3 for single red tubes, or equal to 1 for single blue tubes.

Table~\ref{tab:quot}, first column, shows the realization for the acyclic cubeahedron (range quotient, under addition of red tubes), or acyclic design associahedron, of the diamond poset. Compare to Table~\ref{tab:diam}, where five upper facets have been removed corresponding to the first column of Figure~\ref{fig:equivd}, and the remaining four upper facets have had their constant maximized accordingly. The acyclic design associahedron (cubeahedron) has $f$-vector (17, 26, 11) as seen on the right of Figure~\ref{fig:quotientdiam}.\\


\begin{table}[h!]
    \centering
    \begin{tabular}{| c | c | c |}
\hline
acyclic & acyclic & order \\
 cubeahedron & composihedron &  interval polytope \\ \hline
    $x_a \ge 1$  &  $x_a+x_b \le 72$  &  $x_a+x_b \le 81$ \\
 $x_b \ge 1$   &  $x_c+x_d \le 72$  &  $x_c+x_d \le 81$ \\
  $x_c \ge 1$ &   $x_a+x_d \le 75$  &  $x_a+x_d \le 81$ \\
   $x_d \ge 1$    &    $x_b+x_c \le 75$  &  $x_b+x_c \le 81$ \\
   $x_a+x_b \ge 3$    &  $x_a+x_b+x_c \le 78$ &  $x_a \ge 1$ \\
   $x_c+x_d \ge 3$  &   $x_b+x_c+x_d \le 78$  & $x_b \ge 1$  \\
$x_a+x_b+x_c+x_d \ge 27$  & $x_a+x_b+x_d \le 78$ & $x_c \ge 1$ \\ 
 $x_a+x_b \le 81$   & $x_a+x_c+x_d \le 78$ & $x_d \ge 1$ \\
$x_c+x_d \le 81$  &  $x_a+x_b+x_c+x_d \le 81$ &  \\
 $x_a+x_d \le 81$  & $x_a \ge 1$ & \\
  $x_b+x_c \le 81$  &  $x_b \ge 1$ & \\
  & $x_c \ge 1$ & \\
  & $x_d \ge 1$ & \\
\underline{~~~~~~~cycle equality~~~~~~~} &\underline{~~~~~~~~~~~~~~~~~~~~~~~~~~~~~~~~~~~}  & \underline{~~~~~~~~~~~~~~~~~~~~~~~~~~~~~~~~~~~} \\ 
$x_a-x_b+x_c-x_d = 0$ &$x_a-x_b+x_c-x_d = 0$& $x_a-x_b+x_c-x_d = 0$\\
&& \\
\hline
\end{tabular}
    \caption{The realizations for the diamond poset acyclic cubeahedron, composihedron, and interval polytope using variables $(x_a,x_b,x_c,x_d)$ with subscripts from the edges of the diagram in Figure~\ref{fig:diamondmulti}. Compare to Table~\ref{tab:diam}, where certain facets have been removed or maximized according to the equivalences in Figure~\ref{fig:equivd}. An image of all the polytopes is in Figure~\ref{fig:quotientdiam}, and the realizations are pictured in Figure~\ref{fig:quotpic}.}
    \label{tab:quot}
\end{table}

Table~\ref{tab:quot}, second column, shows the realization  
 for the acyclic composihedron, or domain quotient (under addition of blue tubes), of the diamond poset. Compare to Table~\ref{tab:diam}, where three lower facets have been removed according to the first column of Figure~\ref{fig:equivd}. The only lower facets remaining already had their constant minimized. The acyclic composihedron has $f$-vector (19, 30, 13) as seen on the left of Figure~\ref{fig:quotientdiam}.


Table~\ref{tab:quot}, third column, shows the realization  
 for  the acyclic interval polytope, the simultaneous domain and range quotient for the diamond poset, using variables $(x_a,x_b,x_c,x_d)$ with subscripts from the edges of the diagram in Figure~\ref{fig:diamondmulti}.
The acyclic interval polytope has $f$-vector (10, 16, 8) as seen on the bottom of Figure~\ref{fig:quotientdiam}.


\begin{figure}
    \centering
    \includegraphics[width=0.4\linewidth]{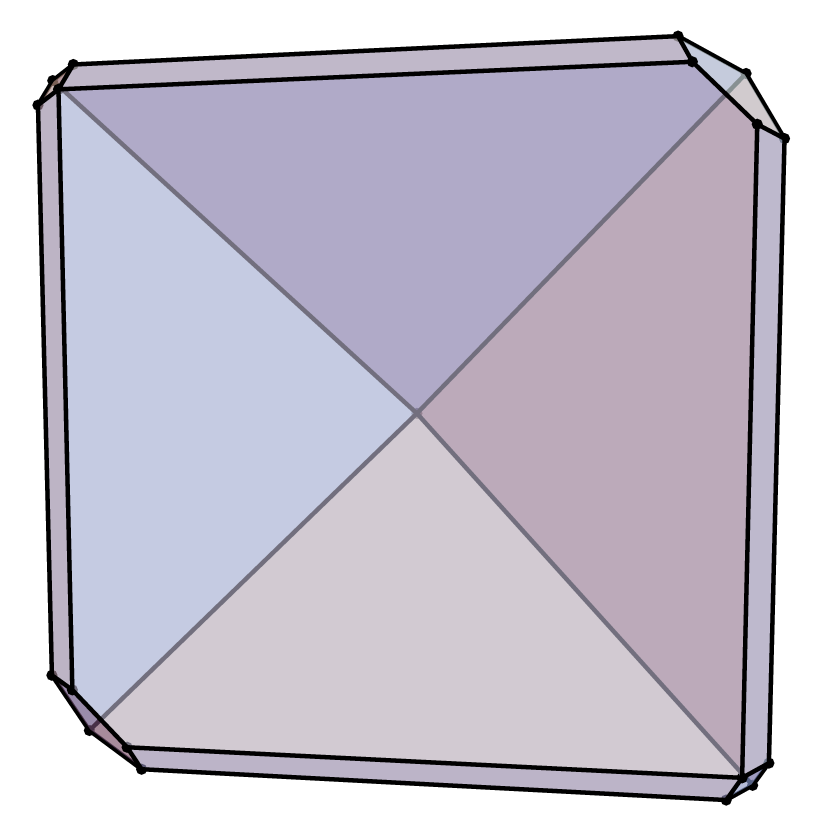}
    \includegraphics[width=0.4\linewidth]{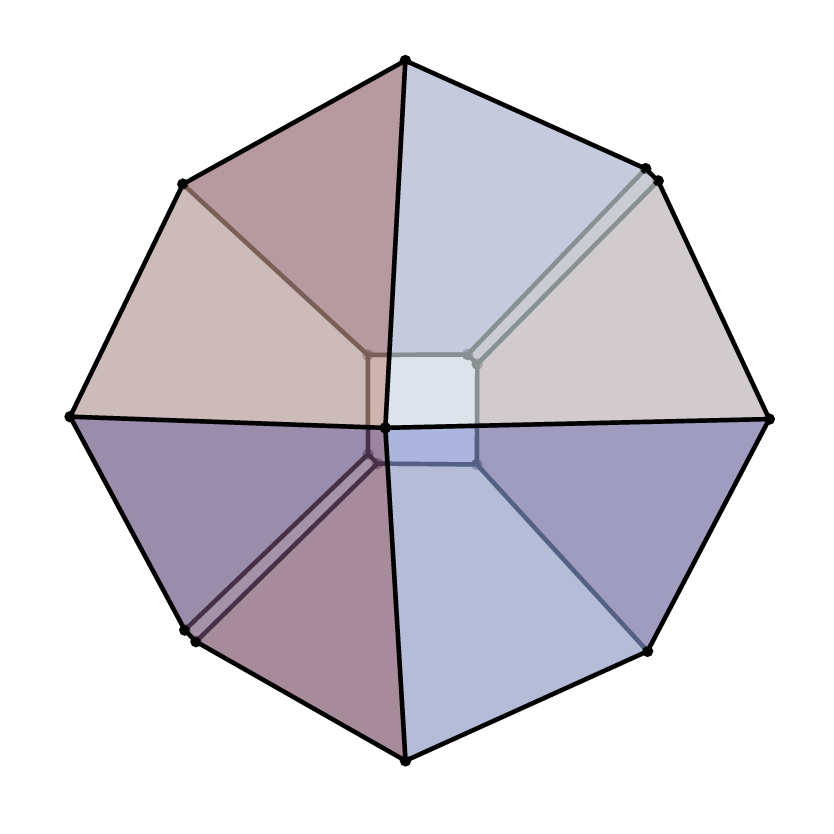}
    \caption{Results of the realization of the acyclic composihedron and cubeahedron (acyclic design associahedron) for the diamond poset, as given in Table~\ref{tab:quot}. }
    \label{fig:quotpic}
\end{figure}

\section*{Acknowledgments}
Thanks to Andrew Sack and to Vincent Pilaud for helpful conversations and encouragement. 

\bibliographystyle{plain}
\bibliography{biblioonew}

@article {forcey,
    AUTHOR = {Forcey, Stefan},
     TITLE = {Convex hull realizations of the multiplihedra},
   JOURNAL = {Topology Appl.},
  FJOURNAL = {Topology and its Applications},
    VOLUME = {156},
      YEAR = {2008},
    NUMBER = {2},
     PAGES = {326--347},
      ISSN = {0166-8641},
   MRCLASS = {52B12 (05A15 18D50 52B70 55P45 55P48)},
  MRNUMBER = {2475119},
MRREVIEWER = {Wolfgang K\"{u}hnel},
       DOI = {10.1016/j.topol.2008.07.010},
       URL = {https://doi.org/10.1016/j.topol.2008.07.010},
}

@article {dev-forc,
    AUTHOR = {Devadoss, Satyan and Forcey, Stefan},
     TITLE = {Marked tubes and the graph multiplihedron},
   JOURNAL = {Algebr. Geom. Topol.},
  FJOURNAL = {Algebraic \& Geometric Topology},
    VOLUME = {8},
      YEAR = {2008},
    NUMBER = {4},
     PAGES = {2081--2108},
      ISSN = {1472-2747},
   MRCLASS = {52B11 (18D50 55P48)},
  MRNUMBER = {2460880},
       DOI = {10.2140/agt.2008.8.2081},
       URL = {https://doi.org/10.2140/agt.2008.8.2081},
}

@article {stasheff,
    AUTHOR = {Stasheff, James Dillon},
     TITLE = {Homotopy associativity of {$H$}-spaces. {I}, {II}},
      NOTE = {{{\bf{1}}08} (1963), 275-292; ibid},
   JOURNAL = {Trans. Amer. Math. Soc.},
  FJOURNAL = {Transactions of the American Mathematical Society},
    VOLUME = {108},
      YEAR = {1963},
     PAGES = {293--312},
      ISSN = {0002-9947},
   MRCLASS = {55.40},
  MRNUMBER = {158400},
MRREVIEWER = {A. H. Clark},
       DOI = {10.1090/s0002-9947-1963-0158400-5},
       URL = {https://doi.org/10.1090/s0002-9947-1963-0158400-5},
}

@article {forc-glew,
    AUTHOR = {Forcey, Stefan and Glew, Ross and Kim, Hyungrok},
     TITLE = {Generalised wavefunction coefficients and
              acyclonesto-cosmohedra},
   JOURNAL = {J. Phys. A},
  FJOURNAL = {Journal of Physics. A. Mathematical and Theoretical},
    VOLUME = {58},
      YEAR = {2025},
    NUMBER = {46},
     PAGES = {Paper No. 465403, 27},
      ISSN = {1751-8113,1751-8121},
      ISBN = {},
   MRCLASS = {52B05 (81U20 85A40)},
  MRNUMBER = {4993140},
       DOI = {10.1088/1751-8121/ae1b4e},
       URL = {https://doi.org/10.1088/1751-8121/ae1b4e},
}

@misc{ardilacosmo,
      title={Combinatorics of the Cosmohedron}, 
      author={Federico Ardila-Mantilla and Nima Arkani-Hamed and Carolina Figueiredo and Francisco Vazão},
      year={2026},
      eprint={2603.03425},
      archivePrefix={arXiv},
      primaryClass={math.CO},
      url={https://arxiv.org/abs/2603.03425}, 
}

@article {multiassoc,
    AUTHOR = {Pilaud, Vincent and Santos, Francisco},
     TITLE = {Multitriangulations as complexes of star polygons},
   JOURNAL = {Discrete Comput. Geom.},
  FJOURNAL = {Discrete \& Computational Geometry. An International Journal
              of Mathematics and Computer Science},
    VOLUME = {41},
      YEAR = {2009},
    NUMBER = {2},
     PAGES = {284--317},
      ISSN = {0179-5376,1432-0444},
   MRCLASS = {52A30 (05C10)},
  MRNUMBER = {2471876},
MRREVIEWER = {Mathieu\ Dutour Sikiri\'{c}},
       DOI = {10.1007/s00454-008-9078-6},
       URL = {https://doi.org/10.1007/s00454-008-9078-6},
}

@misc{bottc,
      title={Constrainahedra}, 
      author={Nathaniel Bottman and Daria Poliakova},
      year={2022},
      eprint={2208.14529},
      archivePrefix={arXiv},
      primaryClass={math.CO},
      url={https://arxiv.org/abs/2208.14529}, 
}

@article {berg,
    AUTHOR = {Benedetti, Carolina and Bergeron, Nantel and Machacek, John},
     TITLE = {Hypergraphic polytopes: combinatorial properties and antipode},
   JOURNAL = {J. Comb.},
  FJOURNAL = {Journal of Combinatorics},
    VOLUME = {10},
      YEAR = {2019},
    NUMBER = {3},
     PAGES = {515--544},
      ISSN = {2156-3527,2150-959X},
   MRCLASS = {05E10 (05C65 52B05)},
  MRNUMBER = {3960512},
       DOI = {10.4310/JOC.2019.v10.n3.a4},
       URL = {https://doi.org/10.4310/JOC.2019.v10.n3.a4},
}

@misc{bottn,
      title={Higher-Categorical Associahedra}, 
      author={Spencer Backman and Nathaniel Bottman and Daria Poliakova},
      year={2024},
      eprint={2409.03633},
      archivePrefix={arXiv},
      primaryClass={math.CO},
      url={https://arxiv.org/abs/2409.03633}, 
}

@Article{dfrs,
 Author = {Devadoss, Satyan Linus and Forcey, Stefan and Reisdorf, Stephen and Showers, Patrick},
 Title = {Convex polytopes from nested posets},
 Journal = {European Journal of Combinatorics},
 ISSN = {0195-6698},
 Volume = {43},
 Pages = {229--248},
 Year = {2015},
 month = jan,
 Language = {English},
 DOI = {10.1016/j.ejc.2014.08.018},
 zbMATH = {6352229},
 Zbl = {1301.05231},
 eprint = { 1306.4208},
 primaryClass = "math",
}

@misc{defant2024operahedronlattices,
      title={Operahedron Lattices}, 
      author={Colin Defant and Andrew Sack},
      year={2024},
      month = feb,
      eprint={2402.12717},
      archivePrefix={arXiv},
      primaryClass={math.CO},
      doi={10.48550/arXiv.2402.12717}, 
}

@article{Glew:2025otn,
    author = "Glew, Ross and {\L{}}ukowski, Tomasz",
    title = "Amplitubes: Graph Cosmohedra",
    eprint = "2502.17564",
    archivePrefix = "arXiv",
    primaryClass = "hep-th",
    year = "2025",
    doi={10.48550/arXiv.2502.17564},
    journal = "Journal of High Energy Physics",
}

@misc{Arkani-Hamed:2024jbp,
    author = "Arkani-Hamed, Nima and Figueiredo, Carolina and Vazão, Francisco",
    title = "{Cosmohedra}",
    eprint = "2412.19881",
    archivePrefix = "arXiv",
    primaryClass = "hep-th",
    month = dec,
    year = "2024",
    doi = {10.48550/arXiv.2412.19881}
}

@article{galashin,
 author = {Galashin{~(\cyrillic{Павел Галашин})}, Pavel},
 title = {{{\(P\)}}-associahedra},
 journal = {Selecta Mathematica, New Series},
 issn = {1022-1824},
 volume = {30},
 number = {1},
 pages = 6,
 year = {2024},
 language = {English},
 doi = {10.1007/s00029-023-00896-1},
 zbMATH = {7782622},
 Zbl = {1529.52012},
 archivePrefix = {arXiv},
       eprint = {2110.07257},
 primaryClass = {math.CO},

}

@ARTICLE{sack,
       author = {Sack, Andrew},
        title = {A realization of poset associahedra},
          doi = {10.48550/arXiv.2301.11449},
archivePrefix = {arXiv},
       eprint = {2301.11449},
 primaryClass = {math.CO},
    JOURNAL = {\foreignlanguage{french}{Séminaire Lotharingien de Combinatoire}},
    VOLUME = {89B},
      YEAR = {2023},
     PAGES = {24},
      ISSN = {1286-4889},
      month = apr,
}

@article {mpp,
    AUTHOR = {Mantovani, Chiara and Padrol, Arnau and Pilaud, Vincent},
     TITLE = {Acyclonestohedra},
   JOURNAL = {\foreignlanguage{french}{Séminaire Lotharingien de Combinatoire}},
    VOLUME = {91B},
      YEAR = {2024},
     PAGES = {28},
      ISSN = {1286-4889},
   MRCLASS = {52B05 (05B35)},
  MRNUMBER = {4818660},
  url = {https://www.mat.univie.ac.at/~slc/wpapers/FPSAC2024/28.pdf},
  month = apr,
}

@article{stanley,
 author = {Stanley, Richard Peter},
 title = {Two poset polytopes},
 journal = {Discrete \& Computational Geometry},
 issn = {0179-5376},
 volume = {1},
 number = 1,
 pages = {9--23},
 year = {1986},
 month = mar,
 language = {English},
 doi = {10.1007/BF02187680},
 zbMATH = {3958127},
 Zbl = {0595.52008}
}

@article{laplante-anfossi,
 author = {Laplante-Anfossi, Guillaume},
 title = {The diagonal of the operahedra},
 journal = {Advances in Mathematics},
 issn = {0001-8708},
 volume = {405},
 month = aug,
 pages = {108494},
 year = {2022},
 language = {English},
 doi = {10.1016/j.aim.2022.108494},
 zbMATH = {7557158},
 Zbl = {1498.52021},
 eprint={2110.14062},
}

@article{cd,
 author = {Carr, Michael and Devadoss, Satyan Linus},
 title = {Coxeter complexes and graph-associahedra},
 journal = {Topology and its Applications},
 issn = {0166-8641},
 volume = {153},
 number = {12},
 pages = {2155--2168},
 year = {2006},
 language = {English},
 doi = {10.1016/j.topol.2005.08.010},
 zbMATH = {5044164},
 Zbl = {1099.52001}
}

@article{postnikov,
 author = {Postnikov{~(\cyrillic{Александр Евгеньевич Постников})}, Alexander Evgenievich},
 title = {Permutohedra, associahedra, and beyond},
 journal = {International Mathematics Research Notices},
 issn = {1073-7928},
 volume = {2009},
 number = {6},
 pages = {1026--1106},
 year = {2009},
 language = {English},
 doi = {10.1093/imrn/rnn153},
 zbMATH = {5542016},
 Zbl = {1162.52007},
 eprint = {math/0507163},
}

@misc{pilaud-sack,
      title={Interval hypergraphic polytopes (or deformed associahedra), Tamari interval posets, and weeping willows}, 
      author={Jose Bastidas and Félix Gélinas and Vincent Pilaud and Germain Poullot and Andrew Sack and Eleni Tzanaki},
      year={2026},
      eprint={2606.18376},
      archivePrefix={arXiv},
      primaryClass={math.CO},
      url={https://arxiv.org/abs/2606.18376}, 
}

@misc{mppfull,
    AUTHOR = {Mantovani, Chiara and Padrol, Arnau and Pilaud, Vincent},
     TITLE = {Facial Nested Complexes and Acyclonestohedra},
      year = {2025},
      eprint={2509.15914},
      archivePrefix={arXiv},
      primaryClass={math.CO},
}

@article {pilaud_berg,
    AUTHOR = {Bergeron, Nantel and Pilaud, Vincent},
     TITLE = {Interval hypergraphic lattices},
   JOURNAL = {European J. Combin.},
  FJOURNAL = {European Journal of Combinatorics},
    VOLUME = {132},
      YEAR = {2026},
    NUMBER = {part B},
     PAGES = {Paper No. 104285, 33},
      ISSN = {0195-6698,1095-9971},
   MRCLASS = {05C65 (05C62 06B05)},
  MRNUMBER = {4993517},
       DOI = {10.1016/j.ejc.2025.104285},
       URL = {https://doi.org/10.1016/j.ejc.2025.104285},
}

@article {doker,
    AUTHOR = {Ardila, Federico and Doker, Jeffrey},
     TITLE = {Lifted generalized permutahedra and composition polynomials},
   JOURNAL = {Adv. in Appl. Math.},
  FJOURNAL = {Advances in Applied Mathematics},
    VOLUME = {50},
      YEAR = {2013},
    NUMBER = {4},
     PAGES = {607--633},
      ISSN = {0196-8858,1090-2074},
   MRCLASS = {52B05 (05A18 05C05 52B11)},
  MRNUMBER = {3032308},
MRREVIEWER = {Ruriko\ Yoshida},
       DOI = {10.1016/j.aam.2013.01.005},
       URL = {https://doi.org/10.1016/j.aam.2013.01.005},
}

@article {bottman,
    AUTHOR = {Bottman, Nathaniel},
     TITLE = {2-associahedra},
   JOURNAL = {Algebr. Geom. Topol.},
  FJOURNAL = {Algebraic \& Geometric Topology},
    VOLUME = {19},
      YEAR = {2019},
    NUMBER = {2},
     PAGES = {743--806},
      ISSN = {1472-2747,1472-2739},
   MRCLASS = {53D37 (52B12)},
  MRNUMBER = {3924177},
       DOI = {10.2140/agt.2019.19.743},
       URL = {https://doi.org/10.2140/agt.2019.19.743},
}

@article {shuffle,
    AUTHOR = {Chapoton, Fr\'{e}d\'{e}ric and Pilaud, Vincent},
     TITLE = {Shuffles of deformed permutahedra, multiplihedra,
              constrainahedra, and biassociahedra},
   JOURNAL = {Ann. H. Lebesgue},
  FJOURNAL = {Annales Henri Lebesgue},
    VOLUME = {7},
      YEAR = {2024},
     PAGES = {1535--1601},
      ISSN = {2644-9463},
   MRCLASS = {52B05 (05A15 05E99 06B99 52B12)},
  MRNUMBER = {4883875},
MRREVIEWER = {Joseph\ Kung},
       DOI = {10.5802/ahl.225},
       URL = {https://doi.org/10.5802/ahl.225},
}

@article {quotient,
    AUTHOR = {Pilaud, Vincent and Santos, Francisco},
     TITLE = {Quotientopes},
   JOURNAL = {Bull. Lond. Math. Soc.},
  FJOURNAL = {Bulletin of the London Mathematical Society},
    VOLUME = {51},
      YEAR = {2019},
    NUMBER = {3},
     PAGES = {406--420},
      ISSN = {0024-6093,1469-2120},
   MRCLASS = {52B05 (06B10 52B12)},
  MRNUMBER = {3964495},
MRREVIEWER = {Egon\ Schulte},
       DOI = {10.1112/blms.12231},
       URL = {https://doi.org/10.1112/blms.12231},
}

\end{document}